\documentclass[smallextended,referee,envcountsect,]{svjour3}

\smartqed

\usepackage{latexsym,amssymb,amsmath}
\usepackage{amsxtra,amscd,enumerate,verbatim,amsfonts,amscd,pstricks}
\usepackage{graphicx}
\usepackage{epstopdf}
\usepackage{framed}

\usepackage[linkcolor=blue, urlcolor=blue, citecolor=blue,
colorlinks, bookmarks]{hyperref}
\journalname{}

\begin{document}

\title{Second-Order Necessary Conditions, Constraint Qualifications and Exact Penalty for Mathematical Programs with Switching Constraints}

\titlerunning{Second-Order Necessary Conditions, CQs and Exact Penalty for MPSC}

\author{Jiawei Chen \and Luyu Liu  \and Yibing Lv  \and   Kequan Zhao }

\institute{
Jiawei Chen \and Luyu Liu
\at
School of Mathematics and Statistics, Southwest University,  Chongqing 400715, China\\
Emails: J.W.Chen713@163.com (J. Chen),\, lyliu124@163.com (L. Liu)\\
Yibing Lv
\at
School of Information and Mathematics, Yangtze University, Jingzhou 434023, China \\
Emails: Yibinglv@yangtzeu.edu.cn\\
Kequan Zhao
\at
School of Mathematical Sciences, Chongqing Normal University,  Chongqing 401331, China\\
Emails: kequanz@163.com
}

\date{Received: date / Accepted: date}

\maketitle

\begin{abstract}
In this paper, we investigate second-order necessary conditions and exact penalty
 of mathematical programs with switching constraints (MPSC). Some new second-order constraint qualifications and second-order quasi-normality are introduced for (MPSC), which are crucial to establish the second-order necessary conditions and the error bound of (MPSC).
We explore the relations among these constraint qualifications in term of (MPSC). The characterizations of Morduhovich stationary point and strong stationary point of (MPSC) are derived under some mild conditions.
 A sufficient condition is provided for a Morduhovich stationary point of (MPSC) being a strong stationary point.
The strong second-order necessary conditions as well as weak second-order necessary conditions
of (MPSC) are established under these weak constraint qualifications.
  Finally,  we obtain the local exact penalty of (MPSC) under the local error bound or some constraint qualifications in term of (MPSC).
\end{abstract}

\keywords{Mathematical programs with switching constraints \and Second-order necessary conditions \and Second-order constraint qualifications \and Exact penalty  \and Error bound}
\subclass{90C30 \and  90C33 \and 90C46}

\section{Introduction}
In recent years, a class of mathematical programming problems with switching constraints (MPSC) has been gained widespread attentions and applications in various fields, such as economic management, optimal control, and mathematical programs with equilibrium constraints; see e.g., \cite{Clason2017,Gugat2008,Hante2013,Seidman2013,Wang2015,JiangZCL,Kanzow2021} and the references therein. Certain equality constraints in (MPSC) can be viewed as the product of two functions, which are called switching constraints.
However, (MPSC) is more challenging to address than the standard nonlinear problems because,
if the switching constraints are treated as equality constraints simply,
some standard constraint qualifications will fail at the feasible points, such as the linear independence constraint qualification (LICQ), Mangasarian-Fromovitz constraint qualification (MFCQ) and Abadie constraint qualification (ACQ); see e.g., \cite{Mehlitz2020}. Therefore, the existing results for standard nonlinear programming problems can not be applied directly to (MPSC). This implies that a locally optimal solution of MPSC may not necessarily be a Karush-Kuhn-Tucker (KKT) point.

For decades, the first-order optimality conditions for (MPSC) have been investigated by the advanced variational analysis. For instance, Mehlitz \cite{Mehlitz2020} introduced the concepts of weak, Mordukhovich, and strong stationarity conditions for (MPSC), and extended some standard constraint qualifications to the  mathematical programs with switching constraints, thereby deriving the first-order optimality conditions for the Mordukhovich stationarity point of (MPSC). Thereafter, Li and Guo \cite{Li2023} introduced the notion of Bouligand stationarity and derived the weakest constraint qualifications for Bouligand and Mordukhovich stationarities of (MPSC) to hold at locally optimal solutions, respectively. Additionally, they extended the constraint qualifications in (MPSC) and explored the relations among existing MPSC-tailored constraint qualifications in \cite{Li2023}. Very recently, Liang and Ye \cite{Liang2021} introduced the concept of $Q$-stationarity for (MPSC) and derived a new optimality condition and local exact penalty results.
As pointed out in \cite{Mehlitz2020,Li2023} that if (MPSC) is regarded as a standard nonlinear programming, then some classical constraint qualifications may fail to hold. It is well-known that constraint qualifications are fundamental notions in the optimization theory such as
 optimality conditions, duality, error bounds, calmness, and penalization; see \cite{Mehlitz2020,Li2023,Ribeiro2023,Chendai,WeiTammerYao,XiaoVanYaoWen} and the references therein.
 So it is necessary to introduce some verifiable new constraint qualifications for (MPSC) including linear independence constraint qualification,
relaxed constant rank constraint qualification and Abadie constraint qualification in the sense of (MPSC).

It is well-known that second-order necessary optimality condition is an important issue in the optimization theory, which is also of vital importance for high order efficient algorithms. Ribeiro and Sachine \cite{Ribeiro2023} studied the strong second-order necessary condition for standard nonlinear problems under the relaxed constant rank constraint qualification, which is considered a relatively weak constraint qualification. The weak constant rank (WCR) condition for standard nonlinear problems was introduced by Andreani et al. \cite{Andreani2007}, which is a very weak condition. Thereafter, Andreani et al. \cite{Andreani2010} derived the weak second-order necessary condition for standard nonlinear problems via the WCR condition. Guo et al. \cite{Guo2013} investigated the second-order optimality conditions for mathematical programming problems with equilibrium constraints (MPEC), introduced the concepts of various second-order optimality conditions such as $S$-multiplier strong/weak second-order necessary conditions, and derived these second-order optimality conditions under different constraint qualifications. For the mathematical programming problems with disjunctive constraints (MPDC), Mehlitz \cite{Mehlitz2020mpdc} introduced the MPDC-tailored version of the linear independence constraint qualification and obtained strong second-order necessary conditions of MPDC under this constraint qualification applying the second-order variational analysis. Observe that the common property of (MPSC) and (MPEC) is orthogonal constraint, and (MPSC) is more general than (MPEC). Compared with the general nonlinear programming problems, the main difficulty of MPEC results from the orthogonal constraints. To the best of our knowledge, there are very little results on the second-order necessary optimality conditions for (MPSC).

Motivated and inspired by the above works, we in this paper investigate
the second-order necessary conditions and exact penalty of (MPSC).
Some constraint qualifications, such as the  MPSC weak constant rank condition, MPSC  piecewise weak constant rank condition, weak/strong second-order constraint qualifications and MPSC piecewise second-order quasi-normality  are introduced in terms of (MPSC). We discuss the relationships among these second-order constraint qualifications and the existing constraint qualifications.  We present a sufficient condition for  a Morduhovich stationary point of (MPSC) being a strong stationary point. Then the strong second-order necessary conditions as well as weak second-order necessary conditions
of (MPSC) are established under these weak constraint qualifications. Last but not least,
 the local exact penalty of (MPSC) is derived under the local error bound or some mild constraint qualifications in term of (MPSC).
 Sufficient conditions for the local error bound of (MPSC) are also considered.

The highlights of this paper are summarized as follows:
\begin{itemize}
  \item A new weak/strong second-order constraint qualification and new weak constant constant rank condition in term of (MPSC) are introduced, which are crucial to establish the weak and strong second-order necessary conditions of (MPSC). Additionally, we explore the relations among the new constraint qualifications and existing ones, indicating that our results promote those in \cite{Guo2013}.

  \item We obtain a sufficient condition for  a Morduhovich stationary point being a strong stationary point.

  \item  The weak/strong second-order necessary conditions
of (MPSC) are established under the weak constraint qualifications.

 \item  The local exact penalty results of (MPSC) are obtained under the local error bound or some mild constraint qualifications in term of (MPSC), which are different from the error bound presented in \cite{Liang2021}.
 We also derive sufficient conditions for the local error bound of (MPSC), which extend the error bound condition for standard nonlinear programming problems in \cite{Bai2023}.
\end{itemize}

The remainder of this paper is organized as follows. In Section \ref{sec2}, we recall some basic definitions and the formulas for various tangent,
normal and linearization cones of cross set, and introduce MPSC weak constant rank condition and MPSC piecewise weak constant rank condition. In Section \ref{sec4:0}, we propose some new second-order constraint qualifications in terms of \eqref{mpsc1} and discuss the relationships among these second-order constraint qualifications and first-order constraint qualifications. In Section \ref{sec3}, we recall the notions of  Mordukhovich and strong stationarity to (MPSC) and give the sufficient conditions for the Mordukhovich stationary point of \eqref{mpsc1} being an strong stationary point. In Section \ref{sec4}, we establish the second-order necessary conditions for (MPSC) under an MPSC-tailored version of constraint qualifications. In Section \ref{sec5}, we obtain the local exact penalty of \eqref{mpsc1} under  the local error bound assumption. We also present the sufficient conditions for the  local error bound.
 Finally, we give the conclusions.

\section{Preliminaries}\label{sec2}

Throughout the paper, unless otherwise specified, let $\mathbb{R}^n$ be $n$-dimensional Euclidean space with the Euclidean norm $\lVert \cdot \rVert $ and inner product $x^\top y$ for two vectors $x,y\in\mathbb{R}^n$, where the superscript $\top$ denotes the transpose. Let
   $\mathbb{B}_\epsilon (\bar{x}):=\{x\in \mathbb{R}^n: \lVert x-\bar{x} \rVert<\epsilon\}$ be the open ball centered at $\bar{x}\in \mathbb{R}^n$  with radius $\epsilon>0$. The notation $x^{k}\xrightarrow{\Omega} \bar{x}$ means that $x^{k}\in \Omega$ for each $k$ and $x^{k}\rightarrow \bar{x}$ as $k\rightarrow\infty$.
   For a nonempty subset $A$ of $\mathbb{R}^n$, the (negative) polar cone of $A$ is defined as
$A^{\circ}:=\{d\in \mathbb{R}^n\,:\, d^\top x\leq 0,\,\forall\, x\in A\}$.
 For a differentiable mapping $\phi: \mathbb{R}^n\rightarrow\mathbb{R}^m$,
 the Jacobian of $\phi$ at $x\in \mathbb{R}^n$ is denoted by $\nabla\phi(x)\in\mathbb{R}^{m\times n}$.
 In particular, $\nabla\phi(x)$ is the gradient of $\phi$ at $x$,
 and $\nabla^2\phi(x)$ is its Hessian matrix at $x$ when $m=1$. For a differentiable function $\varphi: [a,b]\rightarrow\mathbb{R}$,  $\varphi'_+(a)$ signifies the right (Fr\'echet) derivative of $\varphi$ at $a$, where $[a,b]\subseteq\mathbb{R}$.
  We denote by $\textrm{dist}_\mathcal{F}(\bar{x}):=\inf_{x\in \mathcal{F}}\| x-\bar{x} \|$  the distance from a point $\bar{x}\in \mathbb{R}^n$ to a subset $\mathcal{F}$ of $\mathbb{R}^n$.

In this paper, we investigate the following mathematical programs with switching constraints:
\begin{equation}\label{mpsc1}\tag{MPSC}
\begin{aligned}
\left\{\begin{array}{lll}
&\min \quad f\left( x \right) \\
&\mbox{s.t.} \,
	g_i\left( x \right)\leq 0,&i=1,2,\cdots,m,\\
&\quad\,	h_j\left( x \right)=0,&j=1,2,\cdots,p,\\
 &\quad\,   G_k(x)H_k(x)=0,&k=1,2,\cdots,l,
\end{array}\right.
\end{aligned}
\end{equation}
where the functions $f, g_i, h_j, G_k, H_k:\mathbb{R}^n\rightarrow \mathbb{R}$ are twice continuously differentiable for all $i=1,2,\cdots,m, j=1,2,\cdots,p, k=1,2,\cdots,l$.

For the simplicity, let $g:=(g_1,\cdots,g_m)^\top$, $h:=(h_1,\cdots,h_p)^\top$, $G:=(G_1,\cdots,G_l)^\top$ and $H:=(H_1,\cdots,H_l)^\top$.  The Lagrange function associated with \eqref{mpsc1} $\ell$: $\mathbb{R}^n\times\mathbb{R}^m\times\mathbb{R}^p\times\mathbb{R}^l\times\mathbb{R}^l\rightarrow \mathbb{R}$ is defined by
\begin{equation}
  \ell(x,\lambda,\rho,\mu,\nu):=f(x)+\lambda^\top g(x)+\rho^\top h(x)+\mu^\top G(x)+\nu^\top H(x).
\end{equation}
The feasible set of \eqref{mpsc1} is denoted by $\mathcal{F}$.
We also introduce some index sets depended on a feasible point $\bar{x}\in \mathcal{F}$ as follows:
\begin{align*}
 \mathcal{I}_h &:=\{1,2,\cdots,p\},\\
  \mathcal{I}_g &:=\{i\in\{1,2,\cdots,m\}\,:\, g_i(\bar{x})=0\}, \\
  \mathcal{I}_G &:=\left\{i\in\{1,2,\cdots,l\}\,:\, G_i(\bar{x})=0,\,\, H_i(\bar{x})\neq0\right\},\\
  \mathcal{I}_H  &:=\{i\in\{1,2,\cdots,l\}\,:\, G_i(\bar{x})\neq0,\,\, H_i(\bar{x})=0\},\\
  \mathcal{I}_{GH}  &:=\{i\in\{1,2,\cdots,l\}\,:\,G_i(\bar{x})=H_i(\bar{x})=0\}.
\end{align*}
Obviously, $\{\mathcal{I}_G,\mathcal{I}_H,\mathcal{I}_{GH}\}$ is a disjoint partition of $\{1,2,\cdots,l\}$.\vskip2mm

An important approach for dealing with \eqref{mpsc1} is to consider its branching problem. Let $\mathcal{P}(\mathcal{I}_{GH})$ be the set of all disjoint bipartitions of $\mathcal{I}_{GH}$. For any given $(\beta_1,\beta_2)\in\mathcal{P}(\mathcal{I}_{GH})$, a branch problem of \eqref{mpsc1} is defined as follows:
\begin{equation}\label{mpsc3}
  \begin{aligned}
\left\{\begin{array}{lll}
 &  \min \, f\left( x \right) \\
&\mbox{s.t.} \,
	g_i\left( x \right)\leq 0,&i=1,2,\cdots,m,\\
&\quad\,	h_j\left( x \right)=0,&j=1,2,\cdots,p,\\
 &\quad\,   G_k(x)=0,&k\in \mathcal{I}_G\cup\beta_1,\\
&\quad \,   H_k(x)=0,&k\in \mathcal{I}_H\cup\beta_2.
\end{array} \right.
\end{aligned}
\end{equation}
We denote by $\mathcal{F}_{(\beta_1,\beta_2)}$ the feasible set of problem \eqref{mpsc3}.
Clearly, for each $(\beta_1,\beta_2)\in\mathcal{P}(\mathcal{I}_{GH})$,
 the problem \eqref{mpsc3} is a standard nonlinear programming. Besides, $\mathcal{F}= \bigcup_{(\beta_1,\beta_2)\in\mathcal{P}(\mathcal{I}_{GH})}\mathcal{F}_{(\beta_1,\beta_2)}$.

Along with the \eqref{mpsc1}, we present an abstract optimization problem of the form
\begin{align}\label{mpsc2}
  \left\{\begin{array}{lll}
    &\min \,  f(x) \\
    &\mbox{s.t.} \, F(x)\in \mathcal{D},
  \end{array}\right.
\end{align}
where $F$: $\mathbb{R}^n\rightarrow\mathbb{R}^q$ is twice continuously differentiable, and $\mathcal{D}\subseteq \mathbb{R}^q$ is nonempty. If we set $F:=(g,h,\psi)^\top$ and $\mathcal{D}:=(-\infty,0]^m\times \{0\}^p\times \mathcal{S}^l$, where
\begin{equation*}
 \psi:=(G_1,H_1,G_2,H_2,\cdots,G_l,H_l)^\top \textrm{ and } \mathcal{S}:=\{(a,b)\in\mathbb{R}^2\,:\, ab=0\},
\end{equation*}
then the problem \eqref{mpsc2} is reduced to \eqref{mpsc1}. The set $\mathcal{S}:=\{(a,b)\in\mathbb{R}^2\,:\, ab=0\}$ is the so-called cross set, which is nonempty, closed and nonconvex  cone. So, the problem \eqref{mpsc2} is a nonconvex optimization problem.

We next recall some basic concepts and results in variational analysis.

\begin{definition}\cite{Rockafellar1998}
Let $\Omega$ be a nonempty and closed subset of $\mathbb{R}^n$ and $\bar{x}\in \Omega$.
\begin{enumerate}
  \item[{\rm (i)}]
   The tangent cone of $\Omega$ at $\bar{x}$ is defined by
  \begin{equation*}
    \mathcal{T}_\Omega(\bar{x}):=\left\{d\in \mathbb{R}^n:\exists t_k\geq0,x^{k}\xrightarrow{\Omega} \bar{x} \textrm{ s.t. }t_k(x^{k}-\bar{x})\rightarrow d\right\}.
  \end{equation*}
  \item[{\rm (ii)}]
   The Fr\'echet/regular normal cone of $\Omega$ at $\bar{x}$ is defined by $\widehat{\mathcal{N}}_\Omega(\bar{x}):=\mathcal{T}_\Omega(\bar{x})^{\circ}$.
  \item[{\rm (iii)}]
   The Mordukhovich/limiting normal cone of $\Omega$ at $\bar{x}$ is defined by
  \begin{equation*}
    \mathcal{N}_\Omega(\bar{x}):=\left\{d\in \mathbb{R}^n:\exists x^{k}\xrightarrow{\Omega} \bar{x} \textrm{ and } d^{k}\rightarrow d \textrm{ with } d^{k}\in \widehat{\mathcal{N}}_\Omega(x^{k}) \textrm{ for each }k\right\}.
  \end{equation*}
\end{enumerate}
\end{definition}

\begin{remark}{\rm \cite{Rockafellar1998,Mordukhovich2006}} Note that
   $\mathcal{T}_\Omega(\bar{x})$, $\widehat{\mathcal{N}}_\Omega(\bar{x})$ and $\mathcal{N}_\Omega(\bar{x})$ are all closed cones. The Fr\'echet normal cone $\widehat{\mathcal{N}}_\Omega(\bar{x})$ is convex.
   Generally, $\widehat{\mathcal{N}}_\Omega(\bar{x})\subseteq \mathcal{N}_\Omega(\bar{x})$ and equality holds if $\Omega$ is convex, and at this case $\widehat{\mathcal{N}}_\Omega(\bar{x})$ and $\mathcal{N}_\Omega(\bar{x})$ reduce to the convex normal cone defined by
   \begin{equation*}
     \mathcal{N}_\Omega(\bar{x}):=\left\{d\in \mathbb{R}^n \,:\, d^\top(x-\bar{x})\leq 0,\, \forall \, x\in \Omega\right\}.
   \end{equation*}
   Besides,
   we have
  \begin{equation*}
     \limsup_{x^{k}\xrightarrow{\Omega} \bar{x}}\widehat{\mathcal{N}}_\Omega(x^{k})=\mathcal{N}_\Omega(\bar{x}),
  \end{equation*}
where the notion\, \lq\lq $\limsup$\rq\rq\, is Painlev\'e-Kuratowski outer limit (see e.g., \cite[Section4.B]{Rockafellar1998}).
\end{remark}

 In what follows, we proceed to recall some constraint qualifications given in \cite{Mehlitz2020,Li2023,Liang2021}.

\begin{definition}\cite{Mehlitz2020,Li2023,Liang2021}
Let $\bar{x}\in \mathcal{F}$ be a feasible point of \eqref{mpsc1}. We say that
\begin{enumerate}
  \item[{\rm (i)}]
    MPSC linear independence constraint qualification (MPSC-LICQ) holds at $\bar{x}$ iff, the family of gradients
  \begin{equation}\label{licq}
  \left\{\nabla g_i(\bar{x})\right\}_{i\in \mathcal{I}_g}\bigcup \left\{\nabla h_j(\bar{x})\right\}_{j\in \mathcal{I}_h}\bigcup \left\{\nabla G_k(\bar{x})\right\}_{k\in \mathcal{I}_G\cup\mathcal{I}_{GH}}\bigcup\left\{\nabla H_k(\bar{x})\right\}_{k\in \mathcal{I}_H\cup\mathcal{I}_{GH}}
\end{equation}
is linearly independent.
  \item[{\rm (ii)}]
    MPSC relaxed constant rank constraint qualification (MPSC-RCRCQ) holds at $\bar{x}$ iff, there exists $\varepsilon>0$ such that, for any $\mathcal{I}_1\subseteq\mathcal{I}_g$ and $\mathcal{I}_3,\mathcal{I}_4\subseteq\mathcal{I}_{GH}$, the family of gradients
     \begin{equation}\label{rcrcq}
  \left\{\nabla g_i(x)\right\}_{i\in \mathcal{I}_1}\bigcup \left\{\nabla h_j(x)\right\}_{j\in \mathcal{I}_h}\bigcup \left\{\nabla G_k(x)\right\}_{k\in \mathcal{I}_3\cup\mathcal{I}_{G}}\bigcup\left\{\nabla H_k(x)\right\}_{k\in \mathcal{I}_4\cup\mathcal{I}_{H}}
 \end{equation}
  has the same rank for each $x\in \mathbb{B}_\varepsilon(\bar{x})$.

\item [{\rm(iii)}]  MPSC piecewise constant rank of subspace component (MPSC-PCRSC) holds at $\bar{x}$ iff,  for each $(\beta_1,\beta_2)\in\mathcal{P}(\mathcal{I}_{GH})$, the  constant rank of subspace component (CRCS) holds at $\bar{x}$ for nonlinear problem \eqref{mpsc3}.  That is,  for each $(\beta_1,\beta_2)\in\mathcal{P}(\mathcal{I}_{GH})$, there exists $\varepsilon>0$ such that the rank of the family of gradients
\begin{equation}\label{pcrsc}
  \left\{\nabla g_i(x)\right\}_{i\in \mathcal{I}^{-}_g}\bigcup \left\{\nabla h_j(x)\right\}_{j\in \mathcal{I}_h}\bigcup \left\{\nabla G_k(x)\right\}_{k\in \mathcal{I}_G\cup\beta_1}\bigcup\left\{\nabla H_k(x)\right\}_{k\in \mathcal{I}_H\cup\beta_2}
\end{equation}
is constant for all $x\in \mathbb{B}_\varepsilon(\bar{x})$, where
 \begin{align*}
    \begin{array}{ll}
 \mathcal{I}^{-}_g:=   \Bigg\{  l\in \mathcal{I}_g\,:\, - \nabla g_l(\bar{x}) \in & \Bigg\{ \displaystyle \sum_{i\in \mathcal{I}_g \setminus \{l\}}{\lambda_i\nabla g_i(\bar{x})}+\sum_{j=1}^p{\rho_j\nabla h_j(\bar{x})}+\sum_{k\in\mathcal{I}_G\cup\beta_1}{\mu_k\nabla G_k(\bar{x})}\\
 & \,\,\,\, \displaystyle +\sum_{k\in\mathcal{I}_H\cup\beta_2}{\nu_k\nabla H_k(\bar{x})}\,:\,  \lambda_i\geq 0,\,i\in \mathcal{I}_g     \Bigg\}    \Bigg\}.
\end{array}
\end{align*}

  \item[{\rm (iv)}]
   MPSC Abadie constraint qualification (MPSC-ACQ) holds at $\bar{x}$ iff, $\mathcal{L}^{MPSC}_\mathcal{F}(\bar{x})=\mathcal{T}_\mathcal{F}(\bar{x})$.

    \item[{\rm (v)}]
   MPSC Guignard constraint qualification (MPSC-GCQ) holds at $\bar{x}$ iff, $\mathcal{L}^{MPSC}_\mathcal{F}(\bar{x})^{\circ}= \widehat{\mathcal{N}}_\mathcal{F}(\bar{x})$.
\end{enumerate}
\end{definition}

We now introduce a weak constant rank condition of \eqref{mpsc1}, which is slight modification of \cite[Definition 3.4 (iv)]{Guo2013}, which is useful to characterize the second-order necessary optimality conditions and error bound of \eqref{mpsc1}.

\begin{definition}
 Let $\bar{x}\in \mathcal{F}$ be a feasible point of \eqref{mpsc1}. We say that
\begin{enumerate}
  \item[{\rm (i)}]
 MPSC weak constant rank condition (MPSC-WCR) holds at $\bar{x}$ iff, there exists $\varepsilon>0$ such that the family of gradients
\begin{equation}\label{wcr}
  \left\{\nabla g_i(x)\right\}_{i\in \mathcal{I}_g}\bigcup \left\{\nabla h_j(x)\right\}_{j\in \mathcal{I}_h}\bigcup \left\{\nabla G_k(x)\right\}_{k\in \mathcal{I}_G\cup\mathcal{I}_{GH}}\bigcup\left\{\nabla H_k(x)\right\}_{k\in \mathcal{I}_H\cup\mathcal{I}_{GH}}
\end{equation}
has the same rank for all $x\in \mathbb{B}_\varepsilon(\bar{x})$.

 \item[{\rm (ii)}]
    MPSC piecewise WCR (MPSC-PWCR) holds at $\bar{x}$ iff, for each $(\beta_1,\beta_2)\in\mathcal{P}(\mathcal{I}_{GH})$, the WCR holds at $\bar{x}$ for nonlinear problem \eqref{mpsc3}. That is,  for each $(\beta_1,\beta_2)\in\mathcal{P}(\mathcal{I}_{GH})$, there exists $\varepsilon>0$ such that the rank of the family of gradients
\begin{equation}\label{pwcr}
  \left\{\nabla g_i(x)\right\}_{i\in \mathcal{I}_g}\bigcup \left\{\nabla h_j(x)\right\}_{j\in \mathcal{I}_h}\bigcup \left\{\nabla G_k(x)\right\}_{k\in \mathcal{I}_G\cup\beta_1}\bigcup\left\{\nabla H_k(x)\right\}_{k\in \mathcal{I}_H\cup\beta_2}
\end{equation}
is constant for all $x\in \mathbb{B}_\varepsilon(\bar{x})$.

 \end{enumerate}
\end{definition}

\begin{remark}\label{re4.1}
\begin{enumerate}
  \item [{\rm(i)}]
  MPSC-LICQ is a very strong constraint qualification, which implies MPSC-MFCQ, MPSC-NNAMCQ, MPSC-ACQ and MPSC-GCQ \cite{Mehlitz2020}. It is worth noting that the following implication relations holds:
$$\textrm{MPSC-LICQ}\Longrightarrow \textrm{MPSC-RCRCQ}\Longrightarrow\textrm{MPSC-ACQ}\Longrightarrow\textrm{MPSC-GCQ},$$
$$ \textrm{MPSC-RCRCQ}\Longrightarrow\textrm{MPSC-PCRSC},$$
and  $ \textrm{MPSC-RCRCQ}\Longrightarrow\textrm{MPSC-WCR}$.  Together with \cite[Lemma 5.4 and Fig. 1]{Mehlitz2020} yields that
 $$\textrm{MPSC-LICQ}\Longrightarrow\textrm{MPSC-MFCQ}\Longrightarrow\textrm{MPSC-NNAMCQ}\Longrightarrow\textrm{MPSC-ACQ}.$$
  Clearly, if $\mathcal{I}_{GH}=\emptyset$, then MPSC-tailored constraint qualifications above reduce to the classical LICQ, RCRCQ, ACQ and WCR, respectively.

\item [{\rm(ii)}] From \cite[p. 740]{GuoZL}, we have  $\mathcal{I}^{-}_g=  \mathcal{I}_{0}:=\left\{ i\in \mathcal{I}_g \,:\, \nabla g_{i}(\bar{x})^{\top} d=0,\,\forall\,d\in   \mathcal{L}_{\mathcal{F}_{(\beta_1,\beta_2)}}(\bar{x}) \right\}$. Then the MPSC-PCRSC holds at $\bar{x}$ if and only if
    for each $(\beta_1,\beta_2)\in\mathcal{P}(\mathcal{I}_{GH})$, there exists $\varepsilon>0$ such that the rank of the family of gradients
\begin{equation*}
  \left\{\nabla g_i(x)\right\}_{i\in \mathcal{I}_{0}}\bigcup \left\{\nabla h_j(x)\right\}_{j\in \mathcal{I}_h}\bigcup \left\{\nabla G_k(x)\right\}_{k\in \mathcal{I}_G\cup\beta_1}\bigcup\left\{\nabla H_k(x)\right\}_{k\in \mathcal{I}_H\cup\beta_2}
\end{equation*}
is constant for all $x\in \mathbb{B}_\varepsilon(\bar{x})$. In view of this, the MPSC-PCRSC
      can also be called MPSC piecewise relaxed Mangasarian-Fromovitz constraint qualification.

\item [{\rm(iii)}]
 The MPSC-PWCR and the MPSC-WCR are same when $\mathcal{I}_{GH}=\emptyset$.
 Otherwise,
        the MPSC-PWCR and the MPSC-WCR may be not equivalent even if $\mathcal{I}_{GH}$ is singleton; see Examples \ref{Piecewise:CRSC-WCR:2} and \ref{MPSC-PWCR:MPSC-WCR}. If $\mathcal{I}^{-}_g= \mathcal{I}_g$, then the  MPSC-PWCR and the  MPSC-PCRSC are coincided. Moreover, the  MPSC-PWCR and the  MPSC-PCRSC are not equivalent when $\mathcal{I}^{-}_g \neq \mathcal{I}_g$; see Examples \ref{Piecewise:CRSC-WCR:2} and \ref{Piecewise:CRSC-WCR}.
\end{enumerate}
\end{remark}

We next give examples show that the MPSC-PWCR does not imply the MPSC-PCRSC and MPSC-WCR, and that the MPSC-WCR and the MPSC-PWCR are not implied each other.

\begin{example}\label{Piecewise:CRSC-WCR:2}
Consider the problem \eqref{mpsc1} with $f(x):= -x_1, g_{1}(x):=-x_1, g_{2}(x):=x_2$, $g_{3}(x):=x_2-x_{3}^{2}$, $ G_{1}(x):= x_1$  and $ H_{1}(x):= x_3$, where $x=(x_1,x_2,x_3)^\top\in \mathbb{R}^{3}$.
Taking $\bar{x}=(0,0,0)^\top$. Then the index sets $\mathcal{I}_{g}=\{1,2,3\},\mathcal{I}_{GH}=\{1\}$, $\mathcal{P}(\mathcal{I}_{GH})=\{(\emptyset,\{1\}), (\{1\},\emptyset)  \} $ and the gradients
\begin{equation*}
\begin{split}
  &\nabla g_1(x)=(-1,0,0)^\top,\,\nabla g_2(x)=(0,1,0)^\top,\,\nabla g_3(x)=(0,1,-2x_3)^\top,\\
  &\nabla G_1(x)=(1,0,0)^\top,\, \nabla H_1(x)=(0,0,1)^\top.
\end{split}
\end{equation*}
Clearly, the MPSC-WCR holds at $\bar{x}$. For the disjoint bipartition $(\{1\},\emptyset)\in \mathcal{P}(\mathcal{I}_{GH})$, the corresponding problem \eqref{mpsc3} is given as follows:
\begin{equation}\label{bp4}
\begin{split}
  \textrm{min } &  -x_1 \\
  \textrm{s.t. } & -x_1\leq0,\\
                 &  x_2\leq0,\\
                 & x_2-x_3^2\leq0,\\
                 & x_1=0.
  \end{split}
\end{equation}
By simple calculation, we obtain that the linearization cone
$\mathcal{L}_{\mathcal{F}_{(\{1\},\emptyset)}}(\bar{x})=  \{0\} \times (-\mathbb{R}_{+})\times \mathbb{R}$,
 the index set $\mathcal{I}^-_{g}=\{i\in \mathcal{I}_{g}\,:\,\nabla g_i(\bar{x})^\top d=0,\,\forall\, d\in \mathcal{L}_{\mathcal{F}_{(\{1\},\emptyset)}}(\bar{x})\}=\{1\}$. Clearly, the rank of gradients $\{\nabla g_1(x),\nabla G_1(x)\}$ is equal to constant  $1$ around $\bar{x}$. However, the rank of gradients $\{\nabla g_1(x),\nabla g_2(x),\nabla g_3(x),\nabla G_1(x)\}$ is not constant around $\bar{x}$. Consequently, the MPSC-PWCR does not hold at $\bar{x}$, and for the disjoint bipartition $(\{1\},\emptyset)\in \mathcal{P}(\mathcal{I}_{GH})$, the CRSC holds at $\bar{x}$.

Similarly,  for the corresponding problem \eqref{mpsc3} with disjoint bipartition $(\emptyset,\{1\})\in \mathcal{P}(\mathcal{I}_{GH})$,  the linearization cone
$\mathcal{L}_{\mathcal{F}_{(\emptyset,\{1\})}}(\bar{x})=  \mathbb{R}_{+}\times (-\mathbb{R}_{+})\times  \{0\}$ and the index set $\mathcal{I}^-_{g}=\emptyset$.
Clearly, for the disjoint bipartition $(\emptyset,\{1\})\in \mathcal{P}(\mathcal{I}_{GH})$, the CRSC holds at $\bar{x}$, because the rank of gradient $\{\nabla H_1(x)\}$ is equal to constant $1$ for all $x$. Altogether, the MPSC-PCRSC holds at $\bar{x}$.
\end{example}

\begin{example}\label{MPSC-PWCR:MPSC-WCR}
Consider the problem \eqref{mpsc1} with $f(x):= x_1+x_2, g_{1}(x):=x_1$, $H_{1}(x):= x_3$  and $ G_{1}(x):= x_2^2+x_3 $, where $x=(x_1,x_2,x_3)^\top\in \mathbb{R}^{3}$.
Take $\bar{x}=(0,0,0)^\top$. It is easy to see that the gradients $\left\{(1,0,0)^\top,\,(0,2x_2,1)^\top \right\}$  has rank two for all $x$ and $\left\{(1,0,0)^\top,\, (0,0,1)^\top \right\}$ has rank two as well. However, the gradients $\{(1,0,0)^\top,(0,2x_2,1)^\top,(0,0,1)^\top\}$ is not constant around $\bar{x}$. Therefore the MPSC-PWCR holds at $\bar{x}$, but the MPSC-WCR is violated at  $\bar{x}$.
\end{example}

The following example shows that the MPSC-PCRSC is also not implied by the MPSC-PWCR and the MPSC-WCR.

\begin{example}\label{Piecewise:CRSC-WCR}
Consider the problem \eqref{mpsc1} with $f(x):= -x_1, g_{1}(x):=-x_1^{2}, g_{2}(x):=x_2^{2}$, $ G_{1}(x):= x_2-x_1^2$  and $ H_{1}(x):= x_2-x_1^2$, where $x=(x_1,x_2)^\top\in \mathbb{R}^{2}$. Clearly, $\mathcal{F}=\{ (0,0)^\top\}$.
Let $\bar{x}=(0,0)^\top$. Then $\mathcal{I}_{G}=\mathcal{I}_{H}=\emptyset,\mathcal{I}_{GH}=\{1\}$, $\mathcal{P}(\mathcal{I}_{GH})=\{(\emptyset,\{1\}), (\{1\},\emptyset)  \} $ and the gradients
\begin{equation*}
  \nabla g_1(x)=(-1,0)^\top,\nabla g_2(x)=(0,1)^\top,\nabla G_1(x)=(1,0)^\top,\nabla H_1(x)=(-2x_1,1)^\top.
\end{equation*}
For the disjoint bipartition $(\emptyset,\{1\})\in \mathcal{P}(\mathcal{I}_{GH})$, the corresponding problem \eqref{mpsc3} is as follows:
\begin{equation}\label{bp2}
\begin{split}
  \textrm{min } &  -x_1 \\
  \textrm{s.t. } & -x_1\leq0,\\
                 &  x_2\leq0,\\
                 & x_2-x_1^2=0.
  \end{split}
\end{equation}
After calculation, the linearization cone of problem \eqref{bp2} at $\bar{x}$ is
$\mathcal{L}_{\mathcal{F}_{(\emptyset,\{1\})}}(\bar{x})= \mathbb{R}_{+}\times \{0\}$, the index set $\mathcal{I}^{-}_{g}=\{i\in\{1,2\}:\nabla g_i(\bar{x})^\top d=0,\forall d\in \mathcal{L}_{\mathcal{F}_{(\emptyset,\{1\})}}(\bar{x})\}=\{2\}$. It is easy to see that the rank of $\{(0,1)^\top,(-2x_1,1)^\top\}$ is equal to  $1$ when $x_1=0$; otherwise, its  rank is equal to $2$. So, the MPSC-PCRSC fails at $\bar{x}$. However, the MPSC-PWCR holds at $\bar{x}$. As a matter of fact, for the disjoint bipartition $(\emptyset,\{1\})\in \mathcal{P}(\mathcal{I}_{GH})$, the rank of $\left\{(-1,0)^\top,(0,1)^\top,(-2x_1,1)^\top \right\}$ is equal to $2$ for all $x$; and for the disjoint bipartition $(\{1\},\emptyset)\in \mathcal{P}(\mathcal{I}_{GH})$, the rank of $\{(-1,0)^\top,(0,1)^\top,(1,0)^\top\}$ is also equal to $2$ for all $x$.
 Therefore, the MPSC-PCRSC is not implied by the MPSC-PWCR. Besides, the rank of
 $\left\{  \nabla g_1(x)=(-1,0)^\top,\nabla g_2(x)=(0,1)^\top,\nabla G_1(x)=(1,0)^\top,\nabla H_1(x)=(-2x_1,1)^\top
\right\}$ is equal to $2$. So, the MPSC-WCR holds at $\bar{x}$.
\end{example}

The next results present the explicit formulas of tangent cone and normal cone of the cross set  $\mathcal{S}$.

\begin{lemma}{\rm \cite{Mehlitz2020,Li2023,Chendai}}\label{pro2.1}
For any $(a,b)\in \mathcal{S}$, the following formulas are valid
\begin{align*}
  \mathcal{T}_\mathcal{S}(a,b) &=\left\{\begin{array}{ll}
                         \{0\}\times\mathbb{R}, & \textrm{if }a=0\textrm{ and }b\neq0, \\
                         \mathbb{R}\times\{0\}, & \textrm{if }a\neq0\textrm{ and }b=0, \\
                         \mathcal{S}, & \textrm{if }a=b=0
                       \end{array}\right\};\\
   \widehat{\mathcal{N}}_\mathcal{S}(a,b)&=\left\{\begin{array}{ll}
                         \mathbb{R}\times\{0\}, & \textrm{if }a=0\textrm{ and }b\neq0, \\
                         \{0\}\times\mathbb{R}, & \textrm{if }a\neq0\textrm{ and }b=0, \\
                         \{(0,0)\}, & \textrm{if }a=b=0
                       \end{array}\right\}; \\
  \mathcal{N}_\mathcal{S}(a,b)&= \left\{\begin{array}{ll}
                         \mathbb{R}\times\{0\}, & \textrm{if }a=0\textrm{ and }b\neq0, \\
                         \{0\}\times\mathbb{R}, & \textrm{if }a\neq0\textrm{ and }b=0, \\
                         \mathcal{S}, & \textrm{if }a=b=0
                       \end{array}\right\}.
\end{align*}
\end{lemma}

The MPSC-tailored linearization cone to $\mathcal{F}$ at $\bar{x}\in\mathcal{F}$ is defined by
\begin{equation*}
  \mathcal{L}^{MPSC}_\mathcal{F}(\bar{x}):=\left\{d\in\mathbb{R}^n:\nabla F(\bar{x})^\top d\in \mathcal{T}_\mathcal{D}(F(\bar{x}))\right\}.
\end{equation*}
By Lemma \ref{pro2.1}, one can obtain the following explicit representation of $\mathcal{L}^{MPSC}_\mathcal{F}(\bar{x})$
\begin{equation}\label{lcone}
  \begin{split}
     \mathcal{L}^{MPSC}_\mathcal{F}(\bar{x})=\left\{\begin{array}{cc}
                                    d\in\mathbb{R}^n: \begin{array}{ll}
                                                        \nabla g_i(\bar{x})^\top d\leq0, &i\in\mathcal{I}_g, \\
                                                        \nabla h_j(\bar{x})^\top d=0, &j\in\mathcal{I}_h,  \\
                                                        \nabla G_k(\bar{x})^\top d=0, &k\in\mathcal{I}_G,  \\
                                                        \nabla H_k(\bar{x})^\top d=0, &k\in\mathcal{I}_H,  \\
                                                        \nabla G_k(\bar{x})^\top d\cdot\nabla H_k(\bar{x})^\top d=0, &k\in\mathcal{I}_{GH}
                                                      \end{array}
                                    \end{array}\right\}.
   \end{split}
\end{equation}

\begin{lemma} \label{pro2.2}
Let $\bar{x}\in \mathcal{F}$. Then the following formulas are true:
\begin{align*}
  \mathcal{T}_\mathcal{F}(\bar{x})&=\bigcup_{(\beta_1,\beta_2)\in\mathcal{P}(\mathcal{I}_{GH})}{\mathcal{T}_{\mathcal{F}_{(\beta_1,\beta_2)}}(\bar{x})}, \\
  \mathcal{L}^{MPSC}_\mathcal{F}(\bar{x})&=\bigcup_{(\beta_1,\beta_2)\in\mathcal{P}(\mathcal{I}_{GH})}{\mathcal{L}_{\mathcal{F}_{(\beta_1,\beta_2)}}(\bar{x})},
\end{align*}
and  $\mathcal{T}_\mathcal{F}(\bar{x})\subseteq \mathcal{L}^{MPSC}_{\mathcal{F}}(\bar{x})$,
where $\mathcal{L}_{\mathcal{F}_{(\beta_1,\beta_2)}}(\bar{x})$ is the linearization cone of $\mathcal{F}_{(\beta_1,\beta_2)}$ at $\bar{x}$ defined by
\begin{equation*}
  \begin{split}
     \mathcal{L}_{\mathcal{F}_{(\beta_1,\beta_2)}}(\bar{x}):=\left\{\begin{array}{cc}
                                    d\in\mathbb{R}^n: \begin{array}{ll}
                                                        \nabla g_i(\bar{x})^\top d\leq0, &i\in\mathcal{I}_g, \\
                                                        \nabla h_j(\bar{x})^\top d=0, &j\in\mathcal{I}_h,  \\
                                                        \nabla G_k(\bar{x})^\top d=0, &k\in\mathcal{I}_G\cup\beta_1,  \\
                                                        \nabla H_k(\bar{x})^\top d=0, &k\in\mathcal{I}_H\cup\beta_2,
                                                      \end{array}
                                    \end{array}\right\}.
   \end{split}
\end{equation*}
\end{lemma}

 \begin{proof}
  It immediately follows from Lemma 5.1 of \cite{Mehlitz2020}  that
 \begin{align*}
  \mathcal{T}_\mathcal{F}(\bar{x})&=\bigcup_{(\beta_1,\beta_2)\in\mathcal{P}(\mathcal{I}_{GH})}{\mathcal{T}_{\mathcal{F}_{(\beta_1,\beta_2)}}(\bar{x})}, \\
  \mathcal{L}^{MPSC}_\mathcal{F}(\bar{x})&=\bigcup_{(\beta_1,\beta_2)\in\mathcal{P}(\mathcal{I}_{GH})}{ \mathcal{L}_{\mathcal{F}_{(\beta_1,\beta_2)}}(\bar{x})}.
\end{align*}
Note that for any given $(\beta_1,\beta_2)\in\mathcal{P}(\mathcal{I}_{GH})$, $\mathcal{T}_{\mathcal{F}_{(\beta_1,\beta_2)}}(\bar{x})\subseteq \mathcal{L}_{\mathcal{F}_{(\beta_1,\beta_2)}}(\bar{x})$ by the definitions of tangent cone and linearization cone.
Therefore, $\mathcal{T}_\mathcal{F}(\bar{x})\subseteq \mathcal{L}^{MPSC}_{\mathcal{F}}(\bar{x})$. The proof is completed.
 \end{proof}

%


\begin{lemma}{\rm \cite[Lemma 2.2]{Ribeiro2023}}\label{lem4.1}
Let $\sigma$: $\mathbb{R}^n\rightarrow\mathbb{R}^q$ be a twice continuously differentiable mapping. Assume that the rank of the Jacobian $\nabla \sigma(x)$ is a constant for all $x$ in a neighbourhood of the point $\bar{x}\in\mathbb{R}^n$. Set $D:=\ker(\nabla\sigma (\bar{x}))$ the nullspace of $\nabla\sigma (\bar{x})$. Then there exist neighbourhoods $U$, $V$ and a twice continuously differentiable mapping $\phi$: $U\rightarrow V$ such that
\begin{enumerate}
  \item[{\rm (i)}]
   $\phi(\bar{x})=\bar{x}$ and $\nabla \phi(\bar{x})$ is an $n$-order identity matrix.
  \item[{\rm (ii)}]
   $\sigma(\phi(x+d))=\sigma(\phi(x))$ for all $x\in U$ and $d\in D$ satisfying $x+d\in U$.
\end{enumerate}
In particular, $\sigma(\phi(\bar{x}+d))=\sigma(\phi(\bar{x}))=\sigma(\bar{x})$ for all $d\in D$ satisfying $\bar{x}+d\in U$.
\end{lemma}


\section{Constraint qualifications of (MPSC)}\label{sec4:0}
 In this section, we introduce some new second-order constraint qualifications in terms of \eqref{mpsc1} and discuss the relations among these new constraint qualifications and the existing constraint qualifications.

We first recall the notions of critical cone and critical subspace.
Let $\bar{x}\in \mathcal{F}$ be a feasible point of \eqref{mpsc1}. The MPSC-tailored critical cone of $\mathcal{F}$ at $\bar{x}$ is defined by
\begin{equation}\label{ccone1}
  \mathcal{C}^{MPSC}_\mathcal{F}(\bar{x}):=\left\{d\in\mathcal{L}^{MPSC}_\mathcal{F}(\bar{x}):\nabla f(\bar{x})^\top d\leq0\right\}.
\end{equation}
Observe that if $\bar{x}$ is an $S$-stationary point of \eqref{mpsc1}, then due to \eqref{lcone}, for each corresponding multiplier $(\lambda,\rho,\mu,\nu)\in\mathbb{R}^m_{+}\times\mathbb{R}^p\times\mathbb{R}^l\times\mathbb{R}^l$ satisfying \eqref{ss}, the associated critical cone defined by \eqref{ccone1} can be characterized by
\begin{align}
    \mathcal{C}^{MPSC}_\mathcal{F}(\bar{x})&=\left\{d\in \mathcal{L}^{MPSC}_\mathcal{F}(\bar{x})\,:\,\begin{array}{ll}
                                                                                 \nabla g_i(\bar{x})^\top d=0, & i\in \mathcal{I}^+_g(\bar{x},\lambda)
                                                                               \end{array}\right\}\nonumber\\\label{ccone2}
    &=\left\{\begin{array}{cc}
                                    d\in\mathbb{R}^n\,:\, \begin{array}{ll}
                                                        \nabla g_i(\bar{x})^\top d\leq0, &i\in\mathcal{I}_g\setminus\mathcal{I}^+_g(\bar{x},\lambda), \\
                                                        \nabla g_i(\bar{x})^\top d=0, &i\in \mathcal{I}^+_g(\bar{x},\lambda), \\
                                                        \nabla h_j(\bar{x})^\top d=0, &j\in\mathcal{I}_h,  \\
                                                        \nabla G_k(\bar{x})^\top d=0, &k\in\mathcal{I}_G,  \\
                                                        \nabla H_k(\bar{x})^\top d=0, &k\in\mathcal{I}_H,  \\
                                                        \nabla G_k(\bar{x})^\top d\cdot\nabla H_k(\bar{x})^\top d=0, &k\in\mathcal{I}_{GH}
                                                      \end{array}
                                    \end{array}\right\},
\end{align}
where $\mathcal{I}^+_g(\bar{x},\lambda):=\left\{i\in \mathcal{I}_g:\lambda_i>0\right\}$.

Let $\bar{x}\in \mathcal{F}$ be a feasible point of \eqref{mpsc1}. The MPSC critical subspace of $\mathcal{F}$ at $\bar{x}$ is defined by
\begin{equation}\label{cspace1}
\begin{split}
  \tilde{\mathcal{C}}^{MPSC}_\mathcal{F}(\bar{x}):=\left\{\begin{array}{cc}
  d\in\mathbb{R}^n \,:\,                                      \begin{array}{ll}
                                                        \nabla g_i(\bar{x})^\top d=0, &i\in \mathcal{I}_g, \\
                                                        \nabla h_j(\bar{x})^\top d=0, &j\in\mathcal{I}_h,  \\
                                                        \nabla G_k(\bar{x})^\top d=0, &k\in\mathcal{I}_G\cup\mathcal{I}_{GH},  \\
                                                        \nabla H_k(\bar{x})^\top d=0, &k\in\mathcal{I}_H\cup\mathcal{I}_{GH}
                                                      \end{array}
                                    \end{array}\right\}.
\end{split}
\end{equation}

It is easy to see that  $\tilde{\mathcal{C}}^{MPSC}_\mathcal{F}(\bar{x}) \subseteq \mathcal{L}^{MPSC}_\mathcal{F}(\bar{x})$, $\mathcal{C}^{MPSC}_\mathcal{F}(\bar{x}) \subseteq \mathcal{L}^{MPSC}_\mathcal{F}(\bar{x})$ and that
$\tilde{\mathcal{C}}^{MPSC}_\mathcal{F}(\bar{x})\subseteq\mathcal{C}^{MPSC}_\mathcal{F}(\bar{x}) \subseteq \mathcal{L}^{MPSC}_\mathcal{F}(\bar{x})$ when  $\bar{x}\in \mathcal{F}$ is
an $S$-stationary point of \eqref{mpsc1}. Further, if the strict complementarity condition holds at $\bar{x}$ (that is $\mathcal{I}^+_g(\bar{x},\lambda)=\mathcal{I}_g$), then $\tilde{\mathcal{C}}^{MPSC}_\mathcal{F}(\bar{x})=\mathcal{C}^{MPSC}_\mathcal{F}(\bar{x})$.

We next introduce some new second-order constraint qualifications for \eqref{mpsc1}.

\begin{definition}\label{socq}
Let $\bar{x}\in\mathcal{F}$ be a feasible point of \eqref{mpsc1}. The MPSC-tailored strong second-order constraint qualification (MPSC-SSOCQ) holds at $\bar{x}$ iff, for any given nonzero vector $d\in \mathcal{L}^{MPSC}_\mathcal{F}(\bar{x})$, there exist $\delta>0$, sets $\mathcal{J}\subseteq\mathcal{I}_{GH}$, $\mathcal{K}\subseteq\mathcal{I}_{GH}$, and a twice differentiable arc $\xi$: $[0,\delta)\rightarrow \mathbb{R}^n$ such that
 \begin{subequations}
 \begin{align}\label{ssocq1}
          &\xi(t)\in\mathcal{F},\,\forall\, t\in[0,\delta),\\\label{ssocq2}
          &\xi(0)=\bar{x},\quad \xi'_{+}(0)=d, \\\label{ssocq3}
          &g_i(\xi(t))\equiv0, \,\forall\, t\in[0,\delta),\, i\in \mathcal{I},\\\label{ssocq4}
          &G_k(\xi(t))\equiv0,\,\forall\, t\in[0,\delta),\, k\in \mathcal{I}_{G}\cup\mathcal{J},\\\label{ssocq5}
          &H_k(\xi(t))\equiv0,\, \forall\, t\in[0,\delta),\, k\in \mathcal{I}_{H}\cup\mathcal{K},
        \end{align}
 \end{subequations}
where $\mathcal{I}:=\left\{i\in\mathcal{I}_g:\nabla g_i(\bar{x})^\top d=0\right\}$.
\end{definition}

\begin{definition}
Let $\bar{x}\in\mathcal{F}$ be a feasible point of \eqref{mpsc1}. The MPSC-tailored weak second-order constraint qualification (MPSC-WSOCQ) holds at $\bar{x}$ iff, for any given nonzero vector $\tilde{d}\in\tilde{\mathcal{C}}^{MPSC}_\mathcal{F}(\bar{x})$, there exist $\delta>0$, sets $\mathcal{J}\subseteq\mathcal{I}_{GH}$, $\mathcal{K}\subseteq\mathcal{I}_{GH}$, and a twice differentiable arc $\zeta$: $[0,\delta)\rightarrow \mathbb{R}^n$ such that
 \begin{subequations}
 \begin{align}\label{wsocq1}
          &\zeta(t)\in\mathcal{F},\,\forall\, t\in[0,\delta),\\\label{wsocq2}
          &\zeta(0)=\bar{x},\quad \zeta'_{+}(0)=\tilde{d}, \\\label{wsocq3}
          &g_i(\zeta(t))\equiv0, \,\forall\, t\in[0,\delta),\,  i\in \mathcal{I}_g,\\\label{wsocq4}
          &G_k(\zeta(t))\equiv0, \,\forall\, t\in[0,\delta),\,  k\in \mathcal{I}_{G}\cup \mathcal{J},\\\label{wsocq5}
          &H_k(\zeta(t))\equiv0,\, \forall\, t\in[0,\delta),\,  k\in \mathcal{I}_{H}\cup \mathcal{K}.
        \end{align}
 \end{subequations}
\end{definition}

\begin{remark}
\begin{enumerate}
  \item[{\rm (i)}]
From the definition of $\tilde{d}\in \tilde{\mathcal{C}}^{MPSC}_\mathcal{F}(\bar{x})$, we deduce that $i\in \mathcal{I}_g $ implies $i\in \mathcal{I}=\left\{i\in\mathcal{I}_g:\nabla g_i(\bar{x})^\top d=0\right\}$. Taking into account that $\tilde{\mathcal{C}}^{MPSC}_\mathcal{F}(\bar{x})\subseteq \mathcal{L}^{MPSC}_\mathcal{F}(\bar{x})$, it yields that the MPSC-SSOCQ implies the MPSC-WSOCQ.
It is worth noting that MPSC-SSOCQ and MPSC-WSOCQ are constraint qualifications based on geometric concepts which do not depend on the algebraic representation of the feasible set $\mathcal{F}$.

 \item[{\rm (ii)}]
We mention that Maciel et al. \cite{Maciel2011} introduced an analogous (strong) second-order constraint qualification (SOCQ) for constrained vector optimization problems, and established the strong second-order necessary condition under this constraint qualification which guarantee the multipliers associated with the objective functions are all positive. However, we should also mention that the relation \eqref{ssocq3} in SOCQ of \cite{Maciel2011} holds for the whole active set index $\mathcal{I}_{g}$ instead of $\mathcal{I}$.
Besides, MPSC-SSOCQ and MPSC-WSOCQ are motivated by \cite{Guo2013,Fiacco1983,Fiacco1968,Andreani2010s} since the classic strong second-order necessary condition with classical critical cone  was applied to study second-order optimality conditions for nonlinear programming problems and the classic weak second-order necessary condition (WSONC) with classical critical subspace was used to study the second-order practical algorithms.
\end{enumerate}
\end{remark}

We next investigate the relationships among the MPSC-SSOCQ, MPSC-WSOCQ and the MPSC version of constraint qualifications introduced in Section \ref{sec2}.

\begin{lemma}\label{th4.2}
Assume that the MPSC-RCRCQ holds at $\bar{x}\in \mathcal{F}$. Then the MPSC-SSOCQ holds at $\bar{x}$.
\end{lemma}

\begin{proof}
For any  $d\in\mathcal{L}^{MPSC}_\mathcal{F}(\bar{x})$,
let the index sets $\mathcal{J}:= \left\{k\in\mathcal{I}_{GH}\,:\, \nabla G_k(\bar{x})^\top d=0 \right\}$
and $\mathcal{K}:=\left\{k\in\mathcal{I}_{GH} \,:\, \nabla H_k(\bar{x})^\top d=0 \right\}$.
We define a function $\sigma: \mathbb{R}^n\rightarrow\mathbb{R}^{|\mathcal{I}|+|\mathcal{I}_G\cup\mathcal{J}|+|\mathcal{I}_H\cup\mathcal{K}|+p}$ by
\begin{equation*}
  \sigma(x):= \left(g_{\mathcal{I}}(x), G_{\mathcal{I}_G\cup\mathcal{J}}(x), H_{\mathcal{I}_H\cup\mathcal{K}}(x), h(x) \right),
\end{equation*}
where $\mathcal{I}$ was defined in Definition \ref{socq}. So, $d\in \ker \nabla \sigma(\bar{x})$. Since the MPSC-RCRCQ holds at $\bar{x}$, then there exists $\varepsilon>0$ such that the rank of $\nabla \sigma$ is a constant in a neighbourhood $\mathbb{B}_\varepsilon(\bar{x})$ of $\bar{x}$. From Lemma \ref{lem4.1}, we obtain that there exist neighbourhoods $U$, $V$ and a twice continuously differentiable mapping $\phi$: $U\rightarrow V$ such that $\phi(\bar{x})=\bar{x}$, $\nabla \phi(\bar{x})$ is an $n$-order identity matrix and $\sigma(\phi(\bar{x}+d))=\sigma(\bar{x})$ for all $d\in \ker\nabla \sigma(\bar{x}) $ satisfying $\bar{x}+d\in U$. Then there exists $\delta>0$ such that $\bar{x}+td\in U$ for all $t\in [0,\delta)$ and define $\xi$: $[0,\delta)\rightarrow\mathbb{R}^n$ by $\xi(t):=\phi(\bar{x}+td)$. Therefore, we have
\begin{equation*}
  \xi(0)=\phi(\bar{x})=\bar{x} \textrm{ and } \xi'_+(0)=d,
\end{equation*}
which imply \eqref{ssocq2}. Furthermore, one has
\begin{equation*}
\begin{split}
  \left( g_{\mathcal{I}}(\xi(t)),G_{\mathcal{I}_G\cup\mathcal{J}}(\xi(t)),H_{\mathcal{I}_H\cup\mathcal{K}}(\xi(t)),h(\xi(t)) \right)
  =&\sigma(\xi(t))=\sigma(\phi(\bar{x}+td))=\sigma(\bar{x})\\
  =& \left( g_{\mathcal{I}}(\bar{x}),G_{\mathcal{I}_G\cup\mathcal{J}}(\bar{x}),H_{\mathcal{I}_H\cup\mathcal{K}}(\bar{x}),h(\bar{x}) \right)\\
  =&0,\,\,\forall\, t\in[0,\delta),
\end{split}
\end{equation*}
where the third equality follows from $td\in\ker\nabla\sigma(\bar{x})$ for all $t\in[0,\delta)$, this implies the validity of \eqref{ssocq3}-\eqref{ssocq5}.

 Now let us prove the feasibility of the arc $\xi$ for $t>0$ sufficiently small. We consider the following cases:
\begin{enumerate}
  \item[{\rm (i)}]
  For each $k=1,2,\cdots, l$, it follows from $\mathcal{J}\cup\mathcal{K}=\mathcal{I}_{GH}$ that $G_k(\xi(t))H_k(\xi(t))\equiv0$ for arbitrary $t\in[0,\delta)$.
  \item[{\rm (ii)}]
   For $i\notin \mathcal{I}_g$, due to $g_i(\xi(0))=g_i(\bar{x})<0$, it follows straightly from the continuity of $g_i$ and $\xi$ that $g_i(\xi(t))<0$ for all $t\in[0,\delta)$ for $\delta$ small enough.
  \item[{\rm (iii)}]
  For $i\in \mathcal{I}_g\setminus\mathcal{I}$, by the chain rule, we have $(g_i\circ\xi)'_+(0)=\nabla g_i(\bar{x})^\top d<0$. Moreover, we get
  \begin{equation*}
    (g_i\circ\xi)'_+(0)=\lim_{t\rightarrow0^+}\frac{(g_i\circ\xi)(t)-(g_i\circ\xi)(0)}{t}<0,
  \end{equation*}
which implies that $g_i(\xi(t))<g_i(\xi(0))=g_i(\bar{x})=0$ for all $t\in[0,\delta)$, for $\delta$ small enough.
\end{enumerate}
Therefore \eqref{ssocq1} is satisfied. The proof is completed.
\end{proof}

We next are ready to establish the relations among the MPSC-SSOCQ, MPSC-PCRSC and the MPSC-ACQ as well as the MPSC-WCR and the MPSC-WSOCQ.

\begin{lemma}\label{th:SSOCQ-ACQ}
Assume that the MPSC-SSOCQ holds at $\bar{x}\in \mathcal{F}$. Then the MPSC-ACQ holds at $\bar{x}$.
\end{lemma}

\begin{proof}
It follows from Lemma \ref{pro2.2} that the inclusion $\mathcal{T}_\mathcal{F}(\bar{x})\subseteq \mathcal{L}^{MPSC}_\mathcal{F}(\bar{x})$ holds.
Let us show the reverse inclusion $\mathcal{T}_\mathcal{F}(\bar{x})\supseteq \mathcal{L}^{MPSC}_\mathcal{F}(\bar{x})$. Take any nonzero vector $d\in\mathcal{L}^{MPSC}_\mathcal{F}(\bar{x})$, we conclude from the MPSC-SSOCQ assumption that there exist $\delta>0$ and a twice differentiable arc $\xi$: $[0,\delta)\rightarrow \mathbb{R}^n$ such that \eqref{ssocq1} and \eqref{ssocq2} hold. We find a sequence $\{t_k\}_{k\in \mathbb{N}}\subseteq (0,\delta)$ with $t_k\rightarrow 0^+$, and consequently, we have
\begin{equation*}
  d=\xi'_+(0)=\lim_{k\rightarrow\infty}\frac{\xi(t_k)-\xi(0)}{t_k}.
\end{equation*}
Together with Definition \ref{socq} yields that $d\in \mathcal{T}_\mathcal{F}(\bar{x})$. The proof is completed.
\end{proof}

\begin{lemma}\label{MPSC:P:CRSC-ACQ}
  Assume that the MPSC-PCRSC  holds at $\bar{x}\in \mathcal{F}$. Then the  MPSC-ACQ holds at $\bar{x}$.
\end{lemma}

\begin{proof}
 If the MPSC-PCRSC  holds at $\bar{x}\in \mathcal{F}$, then  for each $(\beta_1,\beta_2)\in\mathcal{P}(\mathcal{I}_{GH})$,
 the  constant rank of subspace component (CRCS) holds at $\bar{x}$ for nonlinear problem \eqref{mpsc3}.
 It follows from \cite[Corollary 4.2]{GuoZL} that for each $(\beta_1,\beta_2)\in\mathcal{P}(\mathcal{I}_{GH})$, $\mathcal{T}_{\mathcal{F}_{(\beta_1,\beta_2)}}(\bar{x})= \mathcal{L}_{\mathcal{F}_{(\beta_1,\beta_2)}}(\bar{x})$.
 Together with Lemma \ref{pro2.2} yields that
\begin{align*}
  \mathcal{T}_\mathcal{F}(\bar{x})=  \mathcal{L}^{MPSC}_\mathcal{F}(\bar{x}).
\end{align*}
Consequently, the MPSC-ACQ holds at $\bar{x}$.
The proof is completed.
\end{proof}

\begin{lemma}\label{le:CRSC-SSOCQ}
 Assume that the MPSC-PCRSC holds at $\bar{x}\in \mathcal{F}$. Then the MPSC-SSOCQ holds at $\bar{x}$.
\end{lemma}
	
\begin{proof}
It follows from Lemma \ref{MPSC:P:CRSC-ACQ} that the  MPSC-ACQ holds at $\bar{x}$.
	For any $d\in \mathcal{L}^{MPSC}_\mathcal{F}(\bar{x}) $, using Lemma \ref{pro2.2} yields that there exists a disjoint bipartition $(\beta_1,\beta_2)\in\mathcal{P}(\mathcal{I}_{GH})$ such that $d\in\mathcal{L}_{\mathcal{F}_{(\beta_1,\beta_2)}}(\bar{x})$ and so, $\mathcal{T}_{\mathcal{F}_{(\beta_1,\beta_2)}}(\bar{x})=\mathcal{L}_{\mathcal{F}_{(\beta_1,\beta_2)}}(\bar{x})\ne \emptyset$, which implies $\bar{x}\in\mathcal{F}_{(\beta_1,\beta_2)}$.
	Note that $\beta_1\subseteq\mathcal{J}:=\{k\in \mathcal{I}_{GH}\,:\,\nabla G_k(\bar{x})^\top d=0\}$, $\beta_2\subseteq \mathcal{K}:=\{k\in\mathcal{I}_{GH}\,:\, \nabla H_k(\bar{x})^\top d=0\}$ and $\mathcal{J} \cup \mathcal{K}= \mathcal{I}_{GH}$. Consequently, one can obtain that the MPSC-SSOCQ holds at $\bar{x}$ by the similar argument as that of Lemma \ref{th4.2} by substituting $\beta_1$, $\beta_2$ for $\mathcal{J}$ and $\mathcal{K}$ respectively.
\end{proof}

\begin{lemma}\label{th4.4}
 Assume that the MPSC-WCR holds at $\bar{x}\in \mathcal{F}$.
 Then the MPSC-WSOCQ holds at $\bar{x}$.
\end{lemma}

\begin{proof}
For any $\tilde{d}\in\tilde{\mathcal{C}}^{MPSC}_\mathcal{F}(\bar{x})$, we define $\sigma :\mathbb{R}^n\rightarrow\mathbb{R}^{|\mathcal{I}_g|+|\mathcal{I}_G\cup\mathcal{I}_{GH}|+|\mathcal{I}_H\cup\mathcal{I}_{GH}|+p}$ by
\begin{equation*}
  \sigma(x):=(g_{\mathcal{I}_g}(x),G_{\mathcal{I}_G\cup\mathcal{I}_{GH}}(x),H_{\mathcal{I}_H\cup\mathcal{I}_{GH}}(x),h(x)).
\end{equation*}
Note that $\tilde{\mathcal{C}}^{MPSC}_\mathcal{F}(\bar{x})$ coincides the nullspace of matrix $\nabla \sigma(\bar{x})$, and the MPSC-WCR implies that the rank of $\nabla \sigma$ is a constant in a neighbourhood of $\bar{x}$. It follows from Lemma \ref{lem4.1} that there exist neighbourhoods $U$, $V$ and a twice continuously differentiable mapping $\phi$: $U\rightarrow V$ such that $\phi(\bar{x})=\bar{x}$, $\nabla \phi(\bar{x})$ is an $n$-order identity matrix and $\sigma(\phi(\bar{x}+\tilde{d}))=\sigma(\bar{x})$ for all $\tilde{d}\in \tilde{\mathcal{C}}^{MPSC}_\mathcal{F}(\bar{x})$ satisfying $\bar{x}+\tilde{d}\in U$. Meanwhile, there exists $\delta>0$ such that $\bar{x}+t\tilde{d}\in U$ for all $t\in [0,\delta)$ and define a function $\zeta$: $[0,\delta)\rightarrow\mathbb{R}^n$ by $\zeta(t):=\phi(\bar{x}+t\tilde{d})$. By the similar argument as that of Lemma \ref{th4.2}, we conclude that \eqref{wsocq2} holds, and
\begin{equation*}
  \sigma(\zeta(t))=(g_{\mathcal{I}_g}(\zeta(t)),G_{\mathcal{I}_G\cup\mathcal{I}_{GH}}(\zeta(t)),H_{\mathcal{I}_H\cup\mathcal{I}_{GH}}(\zeta(t)),h(\zeta(t)))\equiv 0,
\end{equation*}
for all $0\leq t<\delta$, which shows that \eqref{wsocq3}-\eqref{wsocq5} hold. In order to prove the arc $\zeta(t)\in\mathcal{F}$ for all $0\leq t<\delta$, it suffices to show $g_i(\zeta(t))<0$ for all $i\notin \mathcal{I}_g$, which can be deduced by the continuity of $g_i$ and $\zeta$ immediately. Consequently, the MPSC-WSOCQ holds at $\bar{x}$. The proof is completed.
\end{proof}

%
%
%
%
%
%
%
%

We next give an example to show that
that Lemma \ref{th4.4} is true and that the MPSC-WCR and MPSC-WSOCQ do not imply the MPSC-ACQ.

\begin{example}
 Consider the problem \eqref{mpsc1} with $f(x):= -x_1+x_3^2, g(x):=-x_1, G(x):=x_1^2$ and $H(x):= x_2-x_1^2$, where $x=(x_1,x_2,x_{3})^\top\in \mathbb{R}^{3}$. Then the feasible set
  $$\mathcal{F}= (\{0\} \times\mathbb{R}\times\mathbb{R})  \cup ( \{ (x_1,x_2)\,:\, x_{2}=x_{1}^{2},\,  x_{1}>0 \} \times\mathbb{R}).$$
Let $\bar{x}=(0,0,0)^{\top}$. After verification,  $\tilde{\mathcal{C}}^{MPSC}_\mathcal{F}(\bar{x})= \{0\}\times \{0\} \times \mathbb{R}$, and the rank of the family of gradients
\begin{equation*}
  \nabla g(x)=(-1,0,0)^\top,\nabla G(x)=(2x_1,0,0)^\top,\nabla H(x)=(-2x_1,1,0)^\top
\end{equation*}
is equal to $2$ around $\bar{x}$. So, the MPSC-WCR holds at $\bar{x}$.  For any $\tilde d=(0,0,\tau)^\top\in \tilde{\mathcal{C}}^{MPSC}_\mathcal{F}(\bar{x})$, set $\zeta(t)=(0,0,\tau t)^\top$ for $0\leq t<1$, it is easy to check that \eqref{wsocq1}-\eqref{wsocq5} are satisfied, and so, the MPSC-WSOCQ holds at $\bar{x}$. However, the MPSC-ACQ is violated at $\bar{x}$. In fact, one can compute that $\mathcal{T}_{\mathcal{F}}(\bar{x})=  \left(\{ 0\}\times\mathbb{R}\times\mathbb{R} \right) \cup \left(\mathbb{R}_{+}\times  \{ 0\}\times\mathbb{R} \right)$ and $\mathcal{L}^{MPSC}_\mathcal{F}(\bar{x})= \mathbb{R}_{+}\times \mathbb{R}\times\mathbb{R}$. So, $\mathcal{T}_{\mathcal{F}}(\bar{x})\neq \mathcal{L}^{MPSC}_\mathcal{F}(\bar{x})$, which implies that the MPSC-ACQ does not hold at $\bar{x}$. Therefore, the MPSC-WCR and MPSC-WSOCQ do not imply the MPSC-ACQ.
\end{example}

The following example also shows that the MPSC-ACQ does not imply the MPSC-WCR.

\begin{example}
 Consider the problem \eqref{mpsc1} with  $f(x):=x_1x_2, g(x):=-x_3^2, G(x):= x_1-x_2$ and $H(x):= x_1+x_2$, where $x=(x_1,x_2,x_{3})^\top\in \mathbb{R}^{3}$.
Let $\bar{x}=(0,0,0)^\top$. It is easy to verify that the MPSC-WCR does not hold at $\bar{x}$ due to the rank of the family of gradients
\begin{equation*}
  \nabla g(x)=(0,0,-2x_3)^\top,\, \nabla G(x)=(1,-1,0)^\top,\,\nabla H(x)=(1,1,0)^\top
\end{equation*}
being not constant around $\bar{x}$. However, the MPSC-ACQ is satisfied at $\bar{x}$. After calculation, the feasible set  $\mathcal{F}=\left\{(x_1,x_2,x_3)^\top\,:\, |x_1|= |x_2|,\, x_3\in\mathbb{R} \right\}$, and
$$\mathcal{L}^{MPSC}_\mathcal{F}(\bar{x})=\left\{ (d_1,d_2,d_3)^\top\,:\, |d_1|=|d_2|,\, d_3\in\mathbb{R} \right\}=\mathcal{T}_{\mathcal{F}}(\bar{x}).$$
Therefore the MPSC-ACQ is not stronger than the MPSC-WCR.
\end{example}

\begin{lemma}\label{PWCR-WSOCQ}
 Assume that the MPSC-PWCR holds at $\bar{x}\in \mathcal{F}$.
 Then the MPSC-WSOCQ holds at $\bar{x}$.
\end{lemma}

\begin{proof}
According to $\bar{x}\in \mathcal{F}=\bigcup_{(\beta_1,\beta_2)\in \mathcal{P}(\mathcal{I}_{GH})} \mathcal{F}_{(\beta_1,\beta_2)}$, we obtain that there exists a disjoint bipartition $(\beta_1,\beta_2)\in \mathcal{P}(\mathcal{I}_{GH})$ such that  $\bar{x}\in \mathcal{F}_{(\beta_1,\beta_2)}$. For any given nonzero $\tilde{d}\in \tilde{\mathcal{C}}^{MPSC}_\mathcal{F}(\bar{x})$, we have $\tilde{d}\in\tilde{\mathcal{C}}^{MPSC}_{\mathcal{F}_{(\beta_1,\beta_2)}}(\bar{x})$	because $\nabla G_k(\bar{x})^\top\tilde{d}=\nabla H_k(\bar{x})^\top\tilde{d}=0$ for all $k\in \mathcal{I}_{GH}$. Since the MPSC-PWCR holds at $\bar{x}$, the WCR holds at $\bar{x}$ for problem \eqref{mpsc3} with the $(\beta_1,\beta_2)$. We define a function $\sigma :\mathbb{R}^n\rightarrow\mathbb{R}^{|\mathcal{I}_g|+|\mathcal{I}_G\cup\beta_1|+|\mathcal{I}_H\cup\beta_2|+p}$ by
\begin{equation*}
  \sigma(x):=(g_{\mathcal{I}_g}(x),G_{\mathcal{I}_G\cup\beta_1}(x),H_{\mathcal{I}_H\cup\beta_2}(x),h(x)).
\end{equation*}
	Note that $\tilde{\mathcal{C}}^{MPSC}_{\mathcal{F}_{(\beta_1,\beta_2)}}(\bar{x})$ is the nullspace of matrix $\nabla \sigma(\bar{x})$. Hence one can conclude that the MPSC-WSOCQ holds at $\bar{x}$ by modifying the argument of Theorem \ref{th4.4}. The proof is completed.
\end{proof}

\section{Relations between Mordukhovich stationary point and strong stationary point of (MPSC)}\label{sec3}

In this section, we mainly study the relations between Mordukhovich stationary point and strong stationary point of \eqref{mpsc1}. In particular, the conditions for  the Mordukhovich stationary point of \eqref{mpsc1} to be strong stationary point are presented under some suitable conditions.

 We first recall the definition of Mordukhovich stationary point and strong stationary point of \eqref{mpsc1}.

\begin{definition}{\rm \cite[Definition 4.1]{Mehlitz2020}}\label{def3.1}
Let $\bar{x}\in \mathcal{F}$ be a feasible point of \eqref{mpsc1}.  $\bar{x}$ is said to be
\begin{enumerate}
  \item[{\rm (i)}]
  Mordukhovich ($M$-) stationary to \eqref{mpsc1} iff, there exists $(\lambda,\rho,\mu,\nu)\in \mathbb{R}^m_{+}\times\mathbb{R}^p\times\mathbb{R}^l\times\mathbb{R}^l$ such that
  \begin{equation}\label{ms}
    \begin{aligned}
    \left\{\begin{array}{ll}
    \nabla_x\ell(\bar{x},\lambda,\rho,\mu,\nu)=0,  \\
     \lambda^\top g(\bar{x})=0, \\
    \mu_k=0, \,\, k\in \mathcal{I}_H, \\
    \nu_k=0, \,\, k\in \mathcal{I}_G,\\
    \mu_k\nu_k=0, \,\, k\in \mathcal{I}_{GH}.
    \end{array}\right.
    \end{aligned}
  \end{equation}

  \item[{\rm (ii)}]
   strong ($S$-) stationary to \eqref{mpsc1} iff, there exists $(\lambda,\rho,\mu,\nu)\in \mathbb{R}^m_{+}\times\mathbb{R}^p\times\mathbb{R}^l\times\mathbb{R}^l$ such that
  \begin{equation}\label{ss}
  \begin{aligned}
    \left\{\begin{array}{ll}
    \nabla_x\ell(\bar{x},\lambda,\rho,\mu,\nu)=0,  \\
     \lambda^\top g(\bar{x})=0, \\
    \mu_k=0, \,\, k\in \mathcal{I}_H, \\
    \nu_k=0, \,\, k\in \mathcal{I}_G,\\
    \mu_k=\nu_k=0, \,\, k\in \mathcal{I}_{GH}.
    \end{array}\right.
    \end{aligned}
  \end{equation}
\end{enumerate}
\end{definition}

\begin{remark}
\begin{enumerate}
  \item[{\rm (i)}]
  It follows from Definition \ref{def3.1} that the following relations hold:
$$\textrm{$S$-stationary}\Longrightarrow \textrm{$M$-stationary}.$$
 The $M$-stationary point of \eqref{mpsc1} is generally not $S$-stationary  point; see Example \ref{ex3.1}.

\item[{\rm (ii)}] The $M$-stationary point and $S$-stationary point of \eqref{mpsc1} are equivalently characterized by the Mordukhovich normal cone and Fr\'echet normal cone, respectively; see \cite[Remark 3.1]{Li2023}.
  \end{enumerate}
\end{remark}

\begin{lemma}{\rm \cite[Remark 3.1]{Li2023}} \label{pro3.1}
Let $\bar{x}\in \mathcal{F}$ be a feasible point of \eqref{mpsc1}, then the following relations are valid
\begin{enumerate}
  \item[{\rm (i)}]
   $\bar{x}$ is $M$-stationary to \eqref{mpsc1} if and only if $0\in \nabla f(\bar{x})+\nabla F(\bar{x})\mathcal{N}_\mathcal{D}(F(\bar{x}))$.
  \item[{\rm (ii)}]
   $\bar{x}$ is $S$-stationary to \eqref{mpsc1} if and only if $0\in \nabla f(\bar{x})+\nabla F(\bar{x})\widehat{\mathcal{N}}_\mathcal{D}(F(\bar{x}))$.
\end{enumerate}
\end{lemma}

The following example shows that the $M$-stationary point of \eqref{mpsc1} may be not $S$-stationary  point.

\begin{example}\label{ex3.1}
Consider the problem \eqref{mpsc1} with $f(x):= x_1-3x_2, g(x):=x_2-x_1, G(x):=x_1, H(x):= x_2$ and $\mathcal{D}:=(-\infty,0]\times \mathcal{S}$, where $x=(x_1,x_2)^\top\in \mathbb{R}^{2}$.
Then $\nabla f(x)= (1,-3)$ and $F(x)=(x_2-x_1,x_1,x_2)^\top$.
It  is easy to verify that $\bar{x}=(0,0)^\top$ is an $M$-stationary point of \eqref{mpsc1}, but $\bar{x}$ is not $S$-stationary point. As a matter of fact, $ g(\bar{x})=0$, and for any $(\lambda,\mu,\nu)\in \mathbb{R}_+\times \mathbb{R}\times \mathbb{R}$ satisfying
\begin{equation*}
  \nabla f(\bar{x})+\lambda \nabla g(\bar{x})+\mu \nabla G(\bar{x})+\nu \nabla H(\bar{x})=0,
\end{equation*}
we have
\begin{equation*}
  1-\lambda+\mu = 0, \,\,
  -3+\lambda+\nu = 0,
\end{equation*}
which implies $\mu+\nu=2$. So, there exists $(\lambda,0,2) \,(\mbox{or},  (\lambda,2,0) ) \in \mathbb{R}_{+} \times \mathbb{R}\times \mathbb{R}$ such that  \eqref{ms} holds.  However, there does not exist $\mu=\nu=0$ and $\lambda \in \mathbb{R}_{+}$ satisfying \eqref{ss}.

On the other hand, it follows from Lemma \ref{pro2.1} that $  \nabla F(\bar{x})\mathcal{N}_\mathcal{D}(F(\bar{x}))=\mathbb{R}^2$ and
\begin{equation*}
  \nabla F(\bar{x}) \widehat{\mathcal{N}}_\mathcal{D}(F(\bar{x}))= \left\{(-t,t)^{\top}\,:\, t\in \mathbb{R}_{+} \right\}\neq (-\infty, 0] \times [0,+\infty).
\end{equation*}
Then
\begin{equation*}
 0\in\nabla f(\bar{x})+ \nabla F(\bar{x})\mathcal{N}_\mathcal{D}(F(\bar{x}))=\mathbb{R}^2,
\end{equation*}
and
\begin{equation*}
 0\notin \nabla f(\bar{x})+ \nabla F(\bar{x}) \widehat{\mathcal{N}}_\mathcal{D}(F(\bar{x}))=  \left\{(1-t, t-3)^{\top}\,:\, t\in \mathbb{R}_{+} \right\}.
\end{equation*}
It therefore implies that Lemma \ref{pro3.1} is true.
\end{example}

The example above shows that $S$-stationary  to \eqref{mpsc1} is strictly stronger than $M$-stationary in some cases. It follows from \cite[Lemma 4.2]{Mehlitz2020}  that for a feasible point $\bar{x}\in \mathcal{F}$ of \eqref{mpsc1}, it is an $M$-stationary point if and only if there is a partition $(\beta_1,\beta_2)\in\mathcal{P}(\mathcal{I}_{GH})$ such that $\bar{x}$ is a KKT point of problem \eqref{mpsc3}.

The next result shows that an $M$-stationary point of \eqref{mpsc1} is also $S$-stationary point  of \eqref{mpsc1}  under appropriate conditions.

\begin{theorem}
Let $\bar{x}\in \mathcal{F}$ be an $M$-stationary point of \eqref{mpsc1} with the associated multipliers $(\lambda,\rho,\mu,\nu)$ satisfying \eqref{ms}.  Define
\begin{align*}
  \beta_1^1:=\left\{k\in \mathcal{I}_{GH}:\mu_k=0\right\},\quad \beta_2^1:=\left\{k\in \mathcal{I}_{GH}:\mu_k\neq0\right\},\\
  \beta_1^2:=\left\{k\in \mathcal{I}_{GH}:\nu_k\neq0\right\},\quad \beta_2^2:=\left\{k\in \mathcal{I}_{GH}:\nu_k=0\right\}.
\end{align*}
Consider the following statements:
\begin{enumerate}
  \item [{\rm (i)}]
  $\bar{x}$ is $S$-stationary to \eqref{mpsc1}.
  \item [{\rm (ii)}]
  For any one of partitions $(\beta_1^i,\beta_2^i)\in\mathcal{P}(\mathcal{I}_{GH})$, $i\in\{1,2\}$, $\bar{x}$ is a KKT point of problem \eqref{mpsc3}.
\end{enumerate}
Then the implication {\rm (i)} $ \Rightarrow$ {\rm (ii)} always holds. Conversely, if the KKT multiplier associated with $\bar{x}$ is $(\lambda,\rho,\mu,\nu)$, that is, for partition $(\beta_1^i,\beta_2^i)\in\mathcal{P}(\mathcal{I}_{GH})$ and $(\lambda,\rho,\mu,\nu)\in \mathbb{R}^{|\mathcal{I}_g|}_+\times\mathbb{R}^p\times\mathbb{R}^{|\mathcal{I}_G\cup \beta_1^i|}\times\mathbb{R}^{|\mathcal{I}_H\cup \beta_2^i|}$, one has
\begin{equation}\label{kkt1}
  \nabla f(\bar{x})+\sum_{i\in \mathcal{I}_g}{\lambda_i\nabla g_{i}(\bar{x})}+\sum_{j=1}^p{\rho_j\nabla h_{j}(\bar{x})}+\sum_{k\in\mathcal{I}_G\cup \beta_1^i}{\mu_k\nabla G_{k}(\bar{x})}+\sum_{k\in\mathcal{I}_H\cup \beta_2^i}{\nu_k\nabla H_{k}(\bar{x})}=0.
\end{equation}
Then the reverse implication {\rm (ii)} $\Rightarrow$ {\rm (i)} holds.
\end{theorem}

\begin{proof}
(i)$\Rightarrow$ (ii): Since $\bar{x}\in \mathcal{F}$ is an $S$-stationary point of \eqref{mpsc1}, it follows from Definition \ref{def3.1} that there exists $(\hat{\lambda},\hat{\rho},\hat{\mu},\hat{\nu})\in \mathbb{R}^m_{+}\times\mathbb{R}^p\times\mathbb{R}^l\times\mathbb{R}^l$ such that \eqref{ss} holds. Obviously, for any partition $(\beta_1,\beta_2)\in\mathcal{P}(\mathcal{I}_{GH})$, we have $\hat{\mu}_k=0$, $k\in \beta_2$ and $\hat{\nu}_k=0$, $k\in \beta_1$.
So, $\bar{x}$ is a KKT point of problem \eqref{mpsc3}.

(ii)$\Rightarrow$ (i): We only prove the case for partition $(\beta_1^1,\beta_2^1)\in\mathcal{P}(\mathcal{I}_{GH})$ since the proof of the case $(\beta_1^2,\beta_2^2)$ is similar. Since $\bar{x}\in \mathcal{F}$ is an $M$-stationary point of MPSC with multipliers $(\lambda,\rho,\mu,\nu)$, then \eqref{ms} holds. We deduce from \eqref{kkt1} that
\begin{align*}
   & \nabla f(\bar{x})+\sum_{i=1}^m{\lambda_i\nabla g_{i}(\bar{x})}+\sum_{j=1}^p{\rho_j\nabla h_{j}(\bar{x})}+\sum_{k\in\mathcal{I}_G}{\mu_k\nabla G_{k}(\bar{x})}+\sum_{k\in\mathcal{I}_H}{\nu_k\nabla H_{k}(\bar{x})}=0,  \\
   & \lambda\ge0, \quad \lambda^\top g(\bar{x})=0.
  \end{align*}
We now set
\begin{equation*}
 \begin{aligned}
  \bar{\mu}_k:=\left\{\begin{array}{ll}
                        \mu_k, & k\in\mathcal{I}_{G}, \\
                        0, & k\notin \mathcal{I}_{G},
                      \end{array}\right.
                      \quad \textrm{and}\quad
   \bar{\nu}_k:=\left\{\begin{array}{ll}
                        \nu_k, & k\in\mathcal{I}_{H}, \\
                        0, & k\notin \mathcal{I}_{H}.
                      \end{array}\right.
\end{aligned}
\end{equation*}
Therefore, the multiplier $(\lambda,\rho,\bar{\mu},\bar{\nu})$ satisfies \eqref{ss}, i.e., $\bar{x}$ is an $S$-stationary point of  \eqref{mpsc1}. The proof is completed.
\end{proof}

It is well-known that for standard nonlinear programming problems, the multipliers that satisfy KKT necessary conditions are uniquely determined when the LICQ holds. Fortunately, for \eqref{mpsc1}, the analogous conclusion \cite[Theorem 4.1]{Mehlitz2020} still holds if the MPSC-LICQ holds. More details about $S$-stationarity can be seen in \cite{Mehlitz2020}.

The following corollary is immediately obtained by \cite[Theorem 5.1]{Mehlitz2020} and Lemma \ref{th:SSOCQ-ACQ}.

\begin{corollary}
Let $\bar{x}\in \mathcal{F}$ be a locally optimal solution of \eqref{mpsc1} and the MPSC-SSOCQ hold at $\bar{x}$. Then $\bar{x}$ is an $M$-stationary point of \eqref{mpsc1}.
\end{corollary}

\section{Second-order necessary conditions of (MPSC)}\label{sec4}
 In this section,  we investigate strong second-order necessary conditions and weak second-order necessary conditions of \eqref{mpsc1} by using the MPSC-tailored versions of second-order constraint qualifications.

\begin{definition}
Let $\bar{x}\in\mathcal{F}$ be an $S$-stationary point of \eqref{mpsc1}. We say that
\begin{enumerate}
  \item[{\rm (i)}]
  the MPSC-tailored strong second-order necessary condition (MPSC-SSONC) holds at $\bar{x}$ iff, for every multiplier vector
$(\lambda,\rho,\mu,\nu)\in\mathbb{R}^m_{+}\times\mathbb{R}^p\times\mathbb{R}^l\times\mathbb{R}^l$ satisfying \eqref{ss} associated with $\bar{x}$, we have
\begin{equation}\label{ssonc}
  d^\top \nabla_{xx}^2 \ell(\bar{x},\lambda,\rho,\mu,\nu)d\geq0,\quad \forall d\in \mathcal{C}^{MPSC}_\mathcal{F}(\bar{x}).
\end{equation}
  \item[{\rm (ii)}]
  the MPSC-tailored weak second-order necessary condition (MPSC-WSONC) holds at $\bar{x}$ iff, for every multiplier vector
$(\lambda,\rho,\mu,\nu)\in\mathbb{R}^m_{+}\times\mathbb{R}^p\times\mathbb{R}^l\times\mathbb{R}^l$ satisfying \eqref{ss} associated with $\bar{x}$, we have
\begin{equation}\label{wsonc}
  \tilde{d}^\top \nabla_{xx}^2 \ell(\bar{x},\lambda,\rho,\mu,\nu)\tilde{d}\geq0,\quad \forall \tilde{d}\in \tilde{\mathcal{C}}^{MPSC}_\mathcal{F}(\bar{x}).
\end{equation}
\end{enumerate}
\end{definition}

\begin{remark}\label{rem4.2}
According to $\tilde{\mathcal{C}}^{MPSC}_\mathcal{F}(\bar{x})\subseteq\mathcal{C}^{MPSC}_\mathcal{F}(\bar{x})$,  we obtain that the MPSC-SSONC implies the MPSC-WSONC. Since $\mathcal{I}^+_g(\bar{x},\lambda)=\mathcal{I}_g$ implies $\tilde{\mathcal{C}}^{MPSC}_\mathcal{F}(\bar{x})=\mathcal{C}^{MPSC}_\mathcal{F}(\bar{x})$, then
 the MPSC-SSONC and the MPSC-WSONC are equivalent when  $\bar{x}$ is
an $S$-stationary point of \eqref{mpsc1} and  $\mathcal{I}^+_g(\bar{x},\lambda)=\mathcal{I}_g$.
\end{remark}

We now investigate the MPSC-SSONC and MPSC-WSONC by the MPSC-SSOCQ and MPSC-WSOCQ, respectively.

\begin{theorem}\label{th4.1}
Let $\bar{x}\in\mathcal{F}$ be a locally optimal solution of \eqref{mpsc1}. Assume that $\bar{x}$ is an $S$-stationary point of \eqref{mpsc1} and MPSC-SSOCQ holds at $\bar{x}$. Then the MPSC-SSONC is satisfied at $\bar{x}$.
\end{theorem}

\begin{proof}
Let $(\lambda,\rho,\mu,\nu)\in\mathbb{R}^m_{+}\times\mathbb{R}^p\times\mathbb{R}^l\times\mathbb{R}^l$ be an arbitrary multiplier vector associated with $\bar{x}$ satisfying \eqref{ss}. For any  $d\in\mathcal{C}^{MPSC}_\mathcal{F}(\bar{x})\subseteq \mathcal{L}^{MPSC}_\mathcal{F}(\bar{x})$ and $d\neq 0$,
 we deduce from the MPSC-SSOCQ that there exist $\delta>0$ and a twice differentiable arc $\xi: [0,\delta)\rightarrow \mathbb{R}^n$ such that $\xi(0)=\bar{x}$, $\xi'_{+}(0)=d$ and $\xi(t)\in\mathcal{F}$ for all $t\in[0,\delta)$. Define a function $\theta: [0,\delta)\rightarrow\mathbb{R}$ by $\theta(t):=f(\xi(t))$. Since $\bar{x}$ is $S$-stationary to \eqref{mpsc1}, by using the chain rule and \eqref{ccone2}, it yields that $\theta'_{+}(0)=\nabla f(\bar{x})^\top d=0$. Since $\bar{x}\in\mathcal{F}$ is a locally optimal solution of \eqref{mpsc1}, without loss of generality, we have
\begin{equation*}
  \theta(0)=f(\bar{x})\leq f(\xi(t))=\theta(t), \,\forall\, t\in [0,\delta).
\end{equation*}
 We conclude from the Mean-Value Theorem that there exists $ \eta_k\in \left(0, \frac{1}{k} \right)$ such that
\begin{equation*}
  \frac{\theta(\frac{1}{k})-\theta(0)}{\frac{1}{k}}=\theta'(\eta_k)
\end{equation*}
for all $k\in\mathbb{N}$ sufficiently large, this implies that $\theta'(\eta_k)\geq0$. Therefore, by using the chain rule for $\theta$, one has
\begin{equation}\label{cr1}
  \theta''_+(0)=d^\top \nabla^2f(\bar{x})d+\nabla f(\bar{x})^\top\xi''_+(0)=\lim_{k\rightarrow\infty}{\frac{\theta'(\eta_k)-\theta'(0)}{\eta_k}}\geq0.
\end{equation}
Due to $g_i(\xi(t))\equiv0 \,(\forall\, i\in\mathcal{I})$, $h_j(\xi(t))\equiv0 \,(\forall\, j\in\mathcal{I}_h)$, $G_k(\xi(t))\equiv0\, (\forall\, k\in\mathcal{I}_G\cup\mathcal{J})$ and $H_k(\xi(t))\equiv0 \,(\forall\, k\in\mathcal{I}_H\cup \mathcal{K})$ for all $t\in[0,\delta)$, using the chain rule directly, we obtain
\begin{align*}
  (g_i\circ\xi)''_+(0) &= d^\top \nabla^2g_i(\bar{x})d+\nabla g_i(\bar{x})^\top\xi''_+(0)=0,\quad \forall i\in\mathcal{I},  \\
  (h_j\circ\xi)''_+(0) &= d^\top \nabla^2h_j(\bar{x})d+\nabla h_j(\bar{x})^\top\xi''_+(0)=0,\quad \forall j\in\mathcal{I}_h,  \\
  (G_k\circ\xi)''_+(0) &= d^\top \nabla^2G_k(\bar{x})d+\nabla G_k(\bar{x})^\top\xi''_+(0)=0,\quad \forall k\in\mathcal{I}_G\cup\mathcal{J},  \\
  (H_k\circ\xi)''_+(0) &= d^\top \nabla^2H_k(\bar{x})d+\nabla H_k(\bar{x})^\top\xi''_+(0)=0,\quad \forall k\in\mathcal{I}_H\cup\mathcal{K}.
\end{align*}
Therefore, multiplying the equalities above by $\lambda_i$, $\rho_j$, $\mu_k$ and $\nu_k$, respectively, summing them together with \eqref{cr1}, and noting that $\mathcal{I}^+_g(\bar{x},\lambda)\subseteq\mathcal{I}$ and $\mathcal{J},\mathcal{K}\subseteq \mathcal{I}_{GH}$, we have
\begin{equation*}
  d^\top \nabla_{xx}^2 \ell(\bar{x},\lambda,\rho,\mu,\nu)d+\nabla_{x} \ell(\bar{x},\lambda,\rho,\mu,\nu)^\top\xi''_+(0)\geq0,
\end{equation*}
which implies that $d^\top \nabla_{xx}^2 \ell(\bar{x},\lambda,\rho,\mu,\nu)d\geq0$ holds. By the arbitrariness of $(\lambda,\rho,\mu,\nu)$ and $d$, we can easily derive that \eqref{ssonc} holds, i.e., MPSC-SSONC holds at $\bar{x}$. The proof is completed.
\end{proof}

\begin{theorem}\label{th4.3}
Let $\bar{x}\in\mathcal{F}$ be a locally optimal solution of \eqref{mpsc1}. Assume that $\bar{x}$ is an $S$-stationary point of \eqref{mpsc1} and the MPSC-WSOCQ holds at $\bar{x}$. Then the MPSC-WSONC holds at $\bar{x}$.
\end{theorem}

\begin{proof}
Let $(\lambda,\rho,\mu,\nu)\in\mathbb{R}^m\times\mathbb{R}^p\times\mathbb{R}^l\times\mathbb{R}^l$ be an arbitrary multiplier vector associated with $\bar{x}$ satisfying \eqref{ss}. For any  $\tilde{d}\in\tilde{\mathcal{C}}^{MPSC}_\mathcal{F}(\bar{x})\setminus \{0\}$, using the MPSC-WSOCQ yields that there exist $\delta>0$ and a twice differentiable arc $\zeta$: $[0,\delta)\rightarrow \mathbb{R}^n$ such that $\zeta(0)=\bar{x}$, $\zeta'_{+}(0)=\tilde{d}$ and $\zeta(t)\in\mathcal{F}$ for all $t\in[0,\delta)$.
 Set $\theta(t):=(f\circ \zeta)(t)$. Then $\theta (\cdot)$ is a twice differentiable function. By the definition of $S$-stationarity and \eqref{ccone2}, we can obtain $\theta'_{+}(0)=\nabla f(\bar{x})^\top \tilde{d}=0$. Since $\bar{x}\in\mathcal{F}$ is a locally optimal solution of \eqref{mpsc1}, then, there exists $k_1\in \mathbb{N}$ such that $\theta(1/k)\geq \theta(0)$ for $k\geq k_1$. By the  Mean-Value Theorem, we deduce that for each $k\geq k_1$,   there exists $ \eta_k\in(0, 1/k)$ such that
\begin{equation*}
  \frac{\theta(1/k)-\theta(0)}{1/k}=\theta'(\eta_k)\geq0.
\end{equation*}
Hence, we obtain
\begin{align*}
 \theta''_+(0)&=(f\circ\zeta)''_+(0) = \tilde{d}^\top \nabla^2f(\bar{x})\tilde{d}+\nabla f(\bar{x})^\top\zeta''_+(0)\geq0,\\
  (g_i\circ\zeta)''_+(0) &= \tilde{d}^\top \nabla^2g_i(\bar{x})\tilde{d}+\nabla g_i(\bar{x})^\top\zeta''_+(0)=0,\,\, \forall\, i\in\mathcal{I}_g,  \\
  (h_j\circ\zeta)''_+(0) &= \tilde{d}^\top \nabla^2h_j(\bar{x})\tilde{d}+\nabla h_j(\bar{x})^\top\zeta''_+(0)=0,\,\, \forall\, j\in\mathcal{I}_h,  \\
  (G_k\circ\zeta)''_+(0) &= \tilde{d}^\top \nabla^2G_k(\bar{x})\tilde{d}+\nabla G_k(\bar{x})^\top\zeta''_+(0)=0,\,\, \forall \, k\in\mathcal{I}_G\cup \mathcal{J},
  \end{align*}
 and
 \begin{align*}
  (H_k\circ\zeta)''_+(0) &= \tilde{d}^\top \nabla^2 H_k(\bar{x})\tilde{d}+\nabla H_k(\bar{x})^\top\zeta''_+(0)=0,\,\, \forall\, k\in\mathcal{I}_H\cup \mathcal{K}.
\end{align*}
Consequently, we obtain that
\begin{equation*}
  \tilde{d}^\top \nabla_{xx}^2 \ell(\bar{x},\lambda,\rho,\mu,\nu)\tilde{d}+\nabla_{x} \ell(\bar{x},\lambda,\rho,\mu,\nu)^\top\zeta''_+(0)\geq0
\end{equation*}
 by the similar process as the proof of Theorem \ref{th4.1}. Taking into account that $\bar{x}$ is an $S$-stationary point of \eqref{mpsc1} with multiplier vector $(\lambda,\rho,\mu,\nu)$, we have that \eqref{wsonc} holds for all $\tilde{d}\in\tilde{\mathcal{C}}^{MPSC}_\mathcal{F}(\bar{x})$ because of $\mu_{k}=\nu_{k}=0$ whenever $k\in \mathcal{I}_{GH}$. Therefore, the MPSC-WSONC holds at $\bar{x}$. The proof is completed.
\end{proof}

The following result can be derived by Remark \ref{re4.1}(i) and
Lemmas \ref{th4.2} and \ref{le:CRSC-SSOCQ} immediately.

\begin{corollary}\label{cor4.1}
Let $\bar{x}\in\mathcal{F}$ be a locally optimal solution of \eqref{mpsc1}. Assume that $\bar{x}$ is an $S$-stationary point of \eqref{mpsc1} and the MPSC-LICQ/MPSC-RCRCQ/MPSC-PCRSC holds at $\bar{x}$. Then the MPSC-SSONC holds at $\bar{x}$.
\end{corollary}

%
%

\begin{remark}
\begin{enumerate}
  \item [{\rm(i)}] Although
  Guo et al. \cite{Guo2013} obtained a very similar result to Corollary \ref{cor4.1} for MPEC, which means that the MPEC-RCRCQ is a second-order constraint qualification for MPEC. However, we utilize a completely different approach and discover that the MPSC-RCRCQ is also a second-order constraint qualification for MPSC.


  \item [{\rm(ii)}]
  If  the switching constraints of MPSC are regarded as simple equality constraints when $\mathcal{I}_{GH}\ne\emptyset$, then the MPSC-SSONC may hold but the classical SSONC does not hold. This indicates that the MPSC-SSONC does not imply classical SSONC when $\mathcal{I}_{GH}\ne\emptyset$; see Example \ref{exam:4.2}.
\end{enumerate}
\end{remark}

The following example illustrates that the assumption about constraint qualifications of Theorem \ref{th4.1} and Corollary \ref{cor4.1} can not be omitted; otherwise, the MPSC-SSONC will do not hold.

\begin{example}
Consider the problem \eqref{mpsc1} with $f(x):= -x_1^2-x_2^2, h(x):=x_1^2-x_2, G_{1}(x):=x_1$, $G_{2}(x):=x_1-x_2^2, H_{1}(x):= x_2$ and $ H_{2}(x):= x_2-x_1^2$, where $x=(x_1,x_2)^\top\in \mathbb{R}^{2}$.
Then  $\mathcal{F}=\left\{(0,0)^\top \right\}$ and $\mathcal{I}_{GH}=\{1,2\}$. Let $\bar{x}:=(0,0)^\top$. So, $\bar{x}$ is certainly the unique optimal solution and $S$-stationary point with an associated multiplier $(0,0,0,0,0)$.
By direct calculation, we obtain
\begin{equation*}
  \mathcal{C}^{MPSC}_\mathcal{F}(\bar{x})=\mathcal{L}^{MPSC}_\mathcal{F}(\bar{x})=\mathbb{R}\times \{ 0\} \subseteq \mathbb{R}^2,
\end{equation*}
 and
\begin{equation*}
  d^\top\nabla^2_{xx}\ell(\bar{x},0,0,0,0,0)d=d^\top\nabla^2f(\bar{x})d=-2d_1^2<0,\,\, \forall\, d\in \mathcal{C}^{MPSC}_\mathcal{F}(\bar{x})\setminus \left\{(0,0)^\top \right\},
\end{equation*}
which means that the MPSC-SSONC does not hold at $\bar{x}$.
For any $d\in\mathcal{L}^{MPSC}_\mathcal{F}(\bar{x})\setminus \left\{(0,0)^\top \right\}$, it is easy to see that there exists no twice differentiable arc $\xi$ satisfying Definition \ref{socq}. As a matter of fact, we can obtain $\xi(t)\equiv(0,0)^\top$ for all $t\in [0,\delta)$ from \eqref{ssocq1}, which is inconsistent with \eqref{ssocq2}. Thus, the MPSC-SSOCQ does not hold at $\bar{x}$. Besides,  we can also check that
the MPSC-LICQ, MPSC-RCRCQ and MPSC-PCRSC do not hold at $\bar{x}$.
\end{example}

 The following example shows that MPSC-SSONC does not imply classical SSONC when $\mathcal{I}_{GH}\ne\emptyset$.

\begin{example}\label{exam:4.2}
Consider the problem \eqref{mpsc1} with $f(x):= -x_1^2-x_2^2, h(x):=x_1-x_2, G(x):=x_1$ and $ H(x):= x_2$, where $x=(x_1,x_2)^\top\in \mathbb{R}^{2}$.
Then  $\mathcal{F}=\left\{(0,0)^\top \right\}$ and $\bar{x}=(0,0)^\top$ is the unique optimal solution and $S$-stationary point of \eqref{mpsc1} with corresponding multiplier $(0,0,0)$. It is easy to see that  $\mathcal{C}^{MPSC}_\mathcal{F}(\bar{x})=\{(0,0)^\top\}$ and the rank of vectors $\{(1,-1)^\top,(1,0)^\top,(0,1)^\top\}$ is equal to $2$ for all $x$ in a neighbourhood of $\bar{x}$. So, the MPSC-RCRCQ holds at $\bar{x}$. After the verification, the MPSC-SSONC is satisfied at $\bar{x}$, i.e.,  Corollary \ref{cor4.1} holds. However, the classical SSONC is violated at $\bar{x}$. Actually, if $h_1(x):=x_1-x_2$ and $h_2(x):=x_1x_2$, then, by direct calculation, the critical cone
\begin{equation*}
  \mathcal{C}_\mathcal{F}(\bar{x})=\left\{(d_1,d_2)^\top\in \mathbb{R}^2:d_1=d_2\right\}.
\end{equation*}
and $(0,0)$ is a KKT multiplier. Picking $\bar{d}=(1,1)^\top\in\mathcal{C}_\mathcal{F}(\bar{x})$, we have
\begin{equation*}
  \bar{d}^\top(\nabla^2f(\bar{x})+0\cdot\nabla^2h_1(\bar{x})+0\cdot\nabla^2h_2(\bar{x}))\bar{d}=-2-2=-4<0,
\end{equation*}
which implies that the classical SSONC does not hold at $\bar{x}$ when the switching constraints are regarded as equality constraints in this problem.
\end{example}

%
%
%
%
%
%

The sufficient conditions for the MPSC-WSONC can be derived by Remark \ref{re4.1}(i) and Lemma \ref{th4.4}.

\begin{corollary}\label{cor4.4}
Let $\bar{x}\in\mathcal{F}$ be a locally optimal solution of \eqref{mpsc1}. Assume that $\bar{x}$ is an $S$-stationary point of \eqref{mpsc1} and the MPSC-WCR holds at $\bar{x}$. Then the MPSC-WSONC holds at $\bar{x}$.
\end{corollary}

\begin{remark}
Theorems \ref{th4.1} and  \ref{th4.3} give the strong second-order necessary conditions and weak second-order necessary conditions for \eqref{mpsc1} under the MPSC-SSOCQ and MPSC-WSOCQ respectively,
which improve Theorems 3.5 and  3.7 of \cite{Guo2013} in one sense even if the second-order optimality conditions for MPEC are discussed in \cite{Guo2013} since Lemmas \ref{th4.2} and \ref{le:CRSC-SSOCQ} imply that the MPSC-SSOCQ is weaker than the MPSC-RCRCQ and the MPSC-PCRSC, and Lemma \ref{th4.4} implies that the MPSC-WSOCQ is weaker than the MPSC-WCR. In other words, we obtained conclusions similar to those of \cite{Guo2013} under the weaker conditions.
\end{remark}

\section{Exact penalty for MPSC}\label{sec5}
In this section, we propose a penalty problem of \eqref{mpsc1} and study the sufficient conditions for the exact penalty by the local error bound conditions and second-order quasi-normality condition, respectively.

We now consider the following penalty problem:
\begin{eqnarray}\label{pp}
 \min f(x)+ \kappa\,\left(\sum_{i=1}^m{[g_i^+(x)]^2}+\sum_{j=1}^p{[h_j(x)]^2}+\sum_{k=1}^l{\min{\left\{[G_k(x)]^2,[H_k(x)]^2\right\}}}\right)^{\frac{1}{2}},
\end{eqnarray}
where $\kappa>0$ is a penalty parameter, and $g_i^+(x):=\max \left\{g_i(x), 0 \right\}$.

We next introduce the exact penalty notion of problem \eqref{pp}.

\begin{definition}\label{exact:penalty}
 We say that problem \eqref{pp} admits a local exact penalization at a locally optimal solution $\bar{x}$ of \eqref{mpsc1} iff, there exists $\bar{\kappa}>0$ such that $\bar{x}$ is a locally optimal solution of problem \eqref{pp} for all $\kappa> \bar{\kappa}$.
\end{definition}

\begin{definition}
The local error bound of \eqref{mpsc1} is satisfied at $\bar{x}\in\mathcal{F}$ iff, there exist $\varepsilon,\alpha>0$ such that
\begin{equation}\label{leb}
  \textrm{dist}_\mathcal{F}(x)\leq \alpha\left(\sum_{i=1}^m{ \left[g_i^+(x) \right]^2}+\sum_{j=1}^p{\left[h_j(x) \right]^2}+\sum_{k=1}^l{\min{\left\{\left[G_k(x) \right]^2, \left[H_k(x) \right]^2\right\}}}\right)^{\frac{1}{2}},
\end{equation}
for all $x\in\mathbb{B}_\varepsilon(\bar{x})$.
\end{definition}

It is worth noting that the local error bound condition \eqref{leb} is also called $\frac{1}{2}$-order local error bound condition or $(\alpha, \frac{1}{2})$-H\"{o}lder error bound condition; see \cite{Bai2023,Kruger2019}.
The following result shows that the locally error bound implies the locally exact penalty of problem \eqref{pp}.

\begin{theorem}\label{exact:penal}
Let $\bar{x}\in\mathcal{F}$ be a locally optimal solution of \eqref{mpsc1}. Assume that the local error bound of \eqref{mpsc1} holds at $\bar{x}$.
 Then the problem \eqref{pp} admits a local exact penalization at $\bar{x}$, i.e., there exists $\bar{\kappa}:=\alpha L_f>0$ such that $\bar{x}$ is a locally optimal solution of problem \eqref{pp} for all $\kappa> \bar{\kappa}$, where $\alpha$ is the error bound constant and $L_f$ is the Lipschitz constant of $f$  around  $\bar{x}$.
\end{theorem}

We now show that the local error bound condition \eqref{leb} holds under some mild constraint qualifications. For this,
we introduce the notion of second-order quasi-normality in the sense of MPSC, which is the extension of the second-order quasi-normality introduced in \cite[Definition 3.2]{Bai2023} from nonlinear problems to switching problems.

\begin{definition}\label{def5.2}
 We say that
  MPSC piecewise second-order quasi-normality (MPSC-PSOQN) holds at $\bar{x}\in \mathcal{F}$ iff, for each $(\beta_1,\beta_2)\in\mathcal{P}(\mathcal{I}_{GH})$, the second-order quasi-normality holds at $\bar{x}$ for nonlinear problem \eqref{mpsc3}. That is, for each $(\beta_1,\beta_2)\in\mathcal{P}(\mathcal{I}_{GH})$, there exists no nonzero vector $(\lambda,\rho,\mu,\nu)$ satisfying $\lambda\geq0$, such that
  \begin{align*}
      & \displaystyle \sum_{i=1}^m{\lambda_i\nabla g_i(\bar{x})}+\sum_{j=1}^p{\rho_j\nabla h_j(\bar{x})}+\sum_{k\in\mathcal{I}_G\cup\beta_1}{\mu_k\nabla G_k(\bar{x})}+\sum_{k\in\mathcal{I}_H\cup\beta_2}{\nu_k\nabla H_k(\bar{x})}=0,\\
     & \displaystyle\tilde{d}^\top \left(\sum_{i=1}^m{\lambda_i\nabla^2 g_i(\bar{x})}+\sum_{j=1}^p{\rho_j\nabla^2 h_j(\bar{x})}+\sum_{k\in\mathcal{I}_G\cup\beta_1}{\mu_k\nabla^2 G_k(\bar{x})}+\sum_{k\in\mathcal{I}_H\cup\beta_2}{\nu_k\nabla^2 H_k(\bar{x})}\right)\tilde{d}\\
     &\geq0,\,\,\forall\, \tilde{d}\in\tilde{\mathcal{C}}^{MPSC}_{\mathcal{F}_{(\beta_1,\beta_2)}}(\bar{x}),
    \end{align*}
     and  there exists a sequence $\{x^s\}$ converging to $\bar{x}$ such that for each $s$,
    \begin{align*}
    \begin{array}{ll}
      \lambda_i>0\Rightarrow \lambda_i g_i(x^s)>0, & \rho_j\ne0\Rightarrow \rho_j h_j(x^s)> 0, \\
      \mu_k\ne0\Rightarrow \mu_k G_k(x^s)>0, & \mu_k\ne0\Rightarrow \mu_k H_k(x^s)>0.
    \end{array}
    \end{align*}
\end{definition}

The following theorem present the sufficient conditions for the local error bound property of \eqref{mpsc1}.

\begin{theorem}\label{th5.2}
 Assume that the MPSC-PWCR and MPSC-PSOQN hold at $\bar{x}\in \mathcal{F}$. Then the local error bound of \eqref{mpsc1} holds at $\bar{x}$.
\end{theorem}

\begin{theorem}\label{th5.3}
 Assume that the MPSC-PCRSC holds at $\bar{x}\in \mathcal{F}$. Then the local error bound of \eqref{mpsc1} holds at $\bar{x}$.
\end{theorem}

\begin{remark}
\begin{enumerate}
  \item [{\rm(i)}] If $\mathcal{I}_{GH}=\emptyset$, Theorem \ref{th5.2} reduces to the error bound results presented in \cite[Theorem 4.1]{Bai2023} for nonlinear programming with equality and inequality constraints. However, if $\mathcal{I}_{GH}\neq \emptyset$, the error bound results presented in \cite[Theorem 4.1]{Bai2023} can not be directly applied to mathematical programs with switching constraints. So, Theorem \ref{th5.2} is an extraordinary extension of the error bound results  \cite{Bai2023}.
  \item [{\rm(ii)}]
  Compared with the error bound condition proposed in \cite{Liang2021} for MPSC, the local error bound condition \eqref{leb} is weaker, since the $\ell_{2}$-norm is smaller than the $\ell_{1}$-norm in Euclidean space.
\end{enumerate}
\end{remark}

The following corollaries follow from Theorems \ref{exact:penal}, \ref{th5.2} and \ref{th5.3} immediately.

\begin{corollary}
Let $\bar{x}\in\mathcal{F}$ be a locally optimal solution of \eqref{mpsc1}. Assume that the MPSC-PWCR and MPSC-PSOQN hold at $\bar{x}$. Then the problem \eqref{pp} admits a local exact penalization at $\bar{x}$.
\end{corollary}

\begin{corollary}
Let $\bar{x}\in\mathcal{F}$ be a locally optimal solution of \eqref{mpsc1}. Assume that the MPSC-PCRSC holds at $\bar{x}$. Then the problem \eqref{pp} admits a local exact penalization at $\bar{x}$.
\end{corollary}

In the end of this section, we summarize the relations among constraint qualifications,  weak/strong second-order necessary conditions, local error bound and local exact penalty for MPSC in Figure \ref{fig1}. These relations are specifically explained in this paper.

\begin{figure}[htbp]
  \centering
  \includegraphics[scale=0.038]{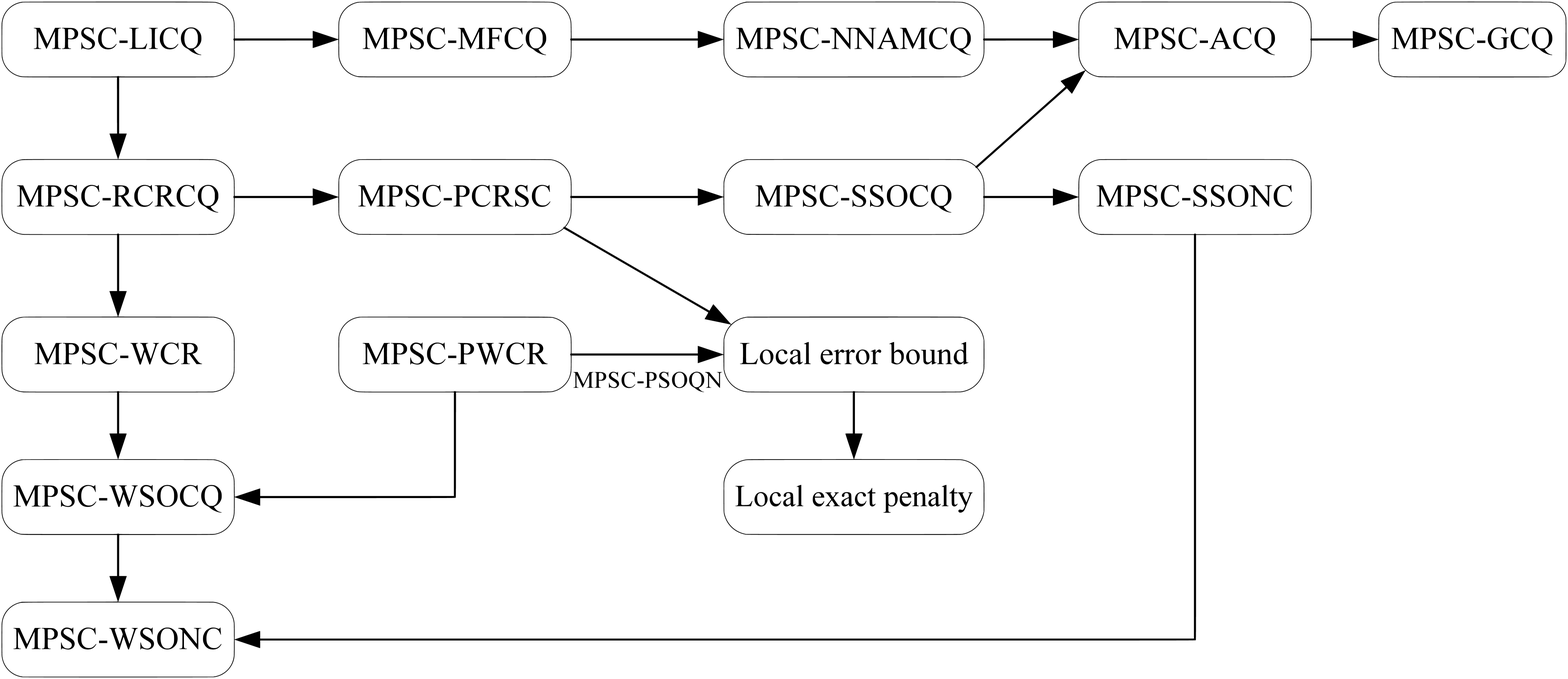}\\
  \caption{Relations among various CQs, W/SSONC, local error bound and local exact penalty for MPSC}\label{fig1}
\end{figure}

\section{Conclusion}\label{sec6}
  Some new MPSC-tailored constraint qualifications, such as MPSC-WSOCQ, MPSC-SSOCQ, MPSC piecewise WCR and MPSC piecewise second-order quasi-normality  for \eqref{mpsc1} are introduced.  Additionally, the relations among these new MPSC-tailored constraint qualifications and some existing constraint qualifications are also discussed. We obtain a sufficient condition for a $M$-stationary point of \eqref{mpsc1} being an $S$-stationary point.
  Moreover,  the weak second-order necessity conditions and strong second-order necessity conditions for \eqref{mpsc1} are established under some  suitable conditions.
 Finally,  the local exact penalization results for \eqref{mpsc1} are derived under the local error bound assumption. The local error bound of \eqref{mpsc1} is proved to be true under the assumptions such as the MPSC-PWCR and MPSC-PSOQN as well as the MPSC-PCRSC.
 As \eqref{mpsc1} is closely related to the MPEC and MPDC, it is interesting to extend MPSC-W/SSOCQ for the MPEC and MPDC. Besides, one may propose an MPDC-tailored version of W/SSOCQ, which would enable us to provide unified results for these optimization problems. It is also interesting to consider the second-order necessity conditions for \eqref{mpsc1} under the $M$-stationary point instead of $S$-stationary point.

\vskip5mm

\noindent{\bf Acknowledgements.}   This paper was supported by  the
   Natural Science Foundation of China (Nos. 12071379, 12271061),
 the Natural Science Foundation of Chongqing(cstc2021jcyj-msxmX0925,  cstc2022ycjh-bgzxm0097), Youth Project of Science and Technology Research Program of Chongqing Education Commission of China (No. KJQN202201802) and the Southwest University Graduate Research Innovation Program (No. SWUS23058). This paper was also supported by the National Key R\&D Program of China (No 2023YFA1011302).

\vskip2mm

\noindent {\bf Author contributions.}
All authors contributed equally to this article.
\vskip2mm

\noindent {\bf Compliance with Ethical Standards}
 It is not applicable.

\vskip2mm
\noindent{\bf Competing interests}\rm

No potential conflict of interest was reported by the authors.

\end{document}